\documentclass[11pt]{article}
\usepackage{epsfig}
\usepackage{amssymb,amsmath,amsthm,amscd}
\usepackage{latexsym}

\newtheorem{thm}{Theorem}[section]
\newtheorem{prop}[thm]{Proposition}
\newtheorem{cor}[thm]{Corollary}
\newtheorem{lem}[thm]{Lemma}

\theoremstyle{definition}
\newtheorem{defn}[thm]{Definition}

\newcounter{labelflag} 

\newcommand{\Label}[1]{
                       \ifnum\thelabelflag=1
                          \ifmmode
                             \makebox[0in][l]{\qquad\fbox{\rm#1}}
                          \else
                             \marginpar{\vspace{0.7\baselineskip}
                                        \hspace{-1.1\textwidth}
                                        \fbox{\rm#1}}
                          \fi
                       \fi
                       \label{#1} }

\newcommand{\be}{\begin{equation}}
\newcommand{\ee}{\end{equation}}

\newcommand{\R}{\mathbb{R}}
\newcommand{\N}{\mathbb{N}}

\def \calf {{  {\mathcal{F}} }}

\def \calftwo {{  {\mathcal{F}}_2  }}
\def \cala {{  {\mathcal{A}}  }}
\def \calb {{  {\mathcal{B}}  }}
\def \cald {{  {\mathcal{D}}  }}
\def \caln {{  {\mathcal{N}}  }}

\def \thonet {{  \theta_{1,t}  }}
\def \thtwot {{  \theta_{2,t}  }}

 \def  \ltwo {{L^2 (\R^n)}}
  \def  \hone {{H^1 (\R^n)}}
  
  \newcommand{\ii}{\int_{\R^n}}

\newcommand{\rh }{\rho ({\frac {|x|^2}{k^2}})}
\newcommand{\rhp }{\rho^\prime ({\frac {|x|^2}{k^2}})}

\begin{document}

\baselineskip=1.3\baselineskip

\begin{titlepage}
\title{\large  \bf \baselineskip=1.3\baselineskip 
 Existence  and Upper Semicontinuity of 
   Attractors
for Stochastic Equations with Deterministic Non-autonomous Terms  }
\vspace{10mm}

\author{ 
Bixiang Wang \vspace{6mm}\\
Department of Mathematics \vspace{1mm}\\
 New Mexico Institute of Mining and Technology  \vspace{1mm}\\ 
Socorro,  NM~87801, USA  \vspace{3mm}\\
Email: bwang@nmt.edu }
\date{}
\end{titlepage}

\maketitle

\medskip

\begin{abstract}  \baselineskip=1.3\baselineskip
We prove   the existence  and uniqueness
of tempered random attractors
for stochastic Reaction-Diffusion
equations  on unbounded domains
 with multiplicative noise
and deterministic  non-autonomous forcing.
We establish    the periodicity of
the  tempered attractors
when   the stochastic  equations 
are forced by periodic functions.
We further prove the upper semicontinuity 
of these attractors when the intensity of 
stochastic perturbations approaches zero. 
 \end{abstract}

semicontinuity;
 unbounded domain.


\section{Introduction}
\setcounter{equation}{0}
      
This paper is concerned with the existence and 
upper semicontinuity of tempered attractors 
for stochastic Reaction-Diffusion equations 
on unbounded domains with deterministic non-autonomous
forcing.  
Given $\tau \in \R$, 
consider    the  following  stochastic equation 
with multiplicative noise  
    defined  on $ (\tau, \infty) \times \R^n $: 
 \be
 \label{intreq}
 {\frac {\partial u}{\partial t}}
  +  \lambda u - \Delta u 
     =  f(x, u)    +  g(t, x )   +  \alpha u \circ   {\frac {d\omega}{dt}},
 \ee
where   $\lambda$  and $\alpha$
 are   positive   numbers,
$g \in L^2_{loc}(\R, \ltwo)$,  
$f$ is a smooth nonlinearity,  and 
$\omega$  is a  two-sided real-valued Wiener process on a probability space.
The  stochastic  equation \eqref{intreq} is understood in the sense
of Stratonovich integration.

 When the deterministic forcing 
 $g$ does not depend on time,
 we can define a random dynamical
 system for equation  \eqref{intreq}
 over a probability space.
 The probability space is responsible
 for the stochastic perturbations and
 can be considered  as a parametric space.
 The existence of random attractors 
 for systems over a single probability space
 has  been investigated
 by many experts in the literature,
 see e.g.,   \cite{arn1, bat1, bat2, car1, car2, 
   car3, car4, car5, chu1, chu2, cra1, cra2, 
    fla1, gar1, gar2, huang1, kloe1, schm1, wan2} 
    and the references therein.
    The reader  is also referred to
      \cite{arn1, dua1, dua2, lia1, moh1} 
    for  the existence  of random  invariant manifolds.
 In this paper,  we  study  
   random attractors of 
  equation \eqref{intreq}
  when the deterministic forcing $g$
  is   time dependent. 
  In this case,   we need to introduce two
  parametric spaces   to 
  describe   the dynamics of the equation:
    one is responsible
  for deterministic 
  forcing  and the other is responsible for 
  stochastic perturbations.  The existence  of
  random attractors for systems over two
  parametric spaces  have  been
  recently  established   in \cite{wan4},
  where the structures of  attractors
  are characterized by random complete solutions. 
  In  the sequel,   we   prove
      the stochastic    equation \eqref{intreq}
  has  a tempered random attractor
  in $\ltwo$  when $g$ is  a   general 
  function depending on time.
  We then show   the tempered attractor
  is actually   periodic in time  if $g$ is    time periodic.
  It is worth noticing   that   the 
  existence of periodic  random attractors 
  was   first   established in  
  \cite{dua3}  for  a model of 
      quasigeostrophic fluid. 
      
        A second  goal of the present paper is
        to prove   the upper semicontinuity 
        of  random attractors  for equation
        \eqref{intreq}   when   the intensity
        $\alpha$ of  noise approaches  zero.  
        This kind of continuity   for  attractors
        has been studied by   many 
        authors, see e.g., 
         \cite{carva1, carva2, hal1, hal2, hal3, tem1}   
         for deterministic  
        attractors, and \cite{car1, lv1, wan5}
        for random attractors
        without  deterministic non-autonomous forcing.
        We here want to  prove the upper
        semicontinuity  of random attractors
        of equation\eqref{intreq} when $g$ is time dependent.
        As far as the author is aware,  this paper  is the first one
        dealing with   continuity  of random attractors
        for stochastic equations  with 
        deterministic non-autonomous forcing.
        We will first  prove  an abstract 
        result and then establish the upper semicontinuity of
        tempered attractors for \eqref{intreq}
        with time dependent  $g$.
         
        Note that the stochastic equation 
        \eqref{intreq} is   defined on the entire
        space $\R^n$.
        Since Sobolev embeddings are not  compact
        on unbounded domains,  we have an extra
        difficulty to prove   the upper semicontinuity 
        of attractors in $\ltwo$.
        We will overcome this difficulty by using  uniform
        estimates on  the tails of functions in  random
        attractors  as in \cite{wan5}.  More precisely,     we prove   
        all functions belonging to  tempered attractors
          are uniformly
        small outside a bounded domain in $\R^n$
        for sufficiently small 
        noise intensity  $\alpha$.

        The outline of the paper is   as follows.
   In the next section, 
   we borrow  some  results  regarding pullback attractors
   for random dynamical systems 
   over two parametric spaces.
   In Section 3,   we prove an abstract result
   on the upper semicontinuity  of pullback attractors
   parametrized by some   variables.
   Section 4 is devoted to the existence of a 
   continuous cocycle in $\ltwo$ for the stochastic
   equation \eqref{intreq}, and Section 5
   contains all uniform estimates
   including those  on the tails of solutions. 
  We  finally  prove the existence and uniqueness
   of tempered attractors   for \eqref{intreq}
   in Section 6, and establish the upper
   semicontinuity of   the  attractors in
   Section 7.

\section{Preliminaries}
\setcounter{equation}{0}

 For the reader\rq{}s convenience, 
 in this section,  we recall  the theory of 
pullback   attractors    for random dynamical
systems  over  two parametric spaces.
All results presented here are not original and can
be found in \cite{wan4}. 
The reader is also referred  to  
  \cite{bat1, cra1, cra2, fla1, schm1} for   random 
attractors  over  one parametric
space, and to \cite{bab1, bal2, hal1, sel1, tem1}
for deterministic attractors.   

Hereafter,  we assume that 
  $(X, d)$   is  a complete
separable  metric space, 
$\Omega_1$ is a nonempty set,
and 
  $(\Omega_2, \calftwo, P)$
is  a probability space.
For every $t \in \R$,  let $\theta_{1,t}: \Omega_1 \to \Omega_1$
be a mapping.  We say $(\Omega_1,  \{\thonet\}_{t \in \R})$ is 
a parametric dynamical system   if  
 $\theta_{1, 0}  $ is the
identity on $\Omega_1$
and $\theta_{1,  s+t}  = \theta_{1,t,}
  \circ \theta_{1,s}  $ for all
$t, s \in \R$. 
Let
 $\theta_2 : \R \times \Omega_2 \to \Omega_2$
 be  a    $(\calb (\R) \times \calftwo, \calftwo)$-measurable mapping.
 We say 
$(\Omega_2, \calftwo, P,  \theta_2)$
is a parametric dynamical system if 
  $\theta_2(0,\cdot) $ is the
identity on $\Omega_2$, 
$\theta_2 (s+t,\cdot) = \theta_2 (t,\cdot) \circ \theta_2 (s,\cdot)$ for all
$t, s \in \R$, 
and $P \theta_2 (t,\cdot)  =P$
for all $t \in \R$.
For convenience,    we
    write 
$\theta_2 (t, \cdot)$ as $\thtwot$ for $t \in \R$, and  write
  $(\Omega_2, \calftwo, P,  \theta_2)$ as 
 $(\Omega_2, \calftwo, P,  \{\thtwot\}_{t \in \R})$.

 \begin{defn}
 Let 
$K: \Omega_1 \times \Omega_2 \to  2^X$  
be  a set-valued mapping.
 We say  $K$ 
   is   measurable
with respect to $\calftwo$
in $\Omega_2$
if  the value  $K(\omega_1, \omega_2)$
is a closed  nonempty subset  of $X$
for all $\omega_1 \in \Omega_1$ and $\omega_2 \in \Omega_2$,
 and  the mapping
$ \omega_2 \in  \Omega_2
 \to d(x, K(\omega_1, \omega_2) )$
is $(  \calftwo, \ \calb(\R) )$-measurable
for every  fixed $x \in X$ and $\omega_1 \in \Omega_1$.
If $K$ is  measurable  with respect to $\calftwo$
in $\Omega_2$,     we   also  say          the family 
$\{K(\omega_1, \omega_2): \omega_1 \in \Omega_1, \omega_2 \in \Omega_2 \}$
  is measurable
with respect to $\calftwo$
 in $\Omega_2$.
 \end{defn}

\begin{defn} \label{ds1}
 Let
$(\Omega_1,  \{\thonet\}_{t \in \R})$
and
$(\Omega_2, \calftwo, P,  \{\thtwot\}_{t \in \R})$
be parametric  dynamical systems.
A mapping $\Phi$: $ \R^+ \times \Omega_1 \times \Omega_2 \times X
\to X$ is called a continuous  cocycle on $X$
over $(\Omega_1,  \{\thonet\}_{t \in \R})$
and
$(\Omega_2, \calftwo, P,  \{\thtwot\}_{t \in \R})$
if   for all
  $\omega_1\in \Omega_1$,
  $\omega_2 \in   \Omega_2 $
  and    $t, \tau \in \R^+$,  the following conditions (i)-(iv)  are satisfied:
\begin{itemize}
\item [(i)]   $\Phi (\cdot, \omega_1, \cdot, \cdot): \R ^+ \times \Omega_2 \times X
\to X$ is
 $(\calb (\R^+)   \times \calftwo \times \calb (X), \
\calb(X))$-measurable;

\item[(ii)]    $\Phi(0, \omega_1, \omega_2, \cdot) $ is the identity on $X$;

\item[(iii)]    $\Phi(t+\tau, \omega_1, \omega_2, \cdot) =
 \Phi(t, \theta_{1,\tau} \omega_1,  \theta_{2,\tau} \omega_2, \cdot) 
 \circ \Phi(\tau, \omega_1, \omega_2, \cdot)$;

\item[(iv)]    $\Phi(t, \omega_1, \omega_2,  \cdot): X \to  X$ is continuous.
    \end{itemize}
    
    If,  in addition,  there exists  a
    positive number   $T $ such that
    for every $t\ge 0$, $\omega_1 \in \Omega_1$  and $\omega_2 \in \Omega_2$,
$$
\Phi(t, \theta_{1, T} \omega_1, \omega_2, \cdot)
= \Phi(t, \omega_1,  \omega_2, \cdot ),
$$
then $\Phi$ is called  
a  continuous periodic  cocycle  on $X$ with period $T$.
\end{defn}

 Throughout the rest of this section,
 we   assume that 
 $\Phi$  is   a continuous  cocycle on $X$
over $(\Omega_1,  \{\thonet\}_{t \in \R})$
and
$(\Omega_2, \calftwo, P,  \{\thtwot\}_{t \in \R})$.
We will  use  $\cald$ to denote
 a  collection  of  some families of  nonempty subsets of $X$:
\be
\label{defcald}
{\cald} = \{ D =\{ D(\omega_1, \omega_2 ) \subseteq X: \ 
D(\omega_1, \omega_2 ) \neq \emptyset,  \ 
  \omega_1 \in \Omega_1, \
  \omega_2 \in \Omega_2\} \}.
\ee
Two elements  $D_1$ and $D_2$  of  $\cald$
are said  to be equal if 
$D_1(\omega_1, \omega_2) =  D_2(\omega_1, \omega_2)$
for   all  $\omega_1 \in \Omega_1$   and
$\omega_2 \in \Omega_2$.

\begin{defn} 
\label{temset} 
Let
$D=\{D(\omega_1, \omega_2): \omega_1 \in \Omega_1, \omega_2 \in \Omega_2 \}$
be a family of    nonempty subsets of $X$. 
We say $D$ is tempered in $X$ 
with respect to $(\Omega_1,  \{\thonet\}_{t \in \R})$
and
$(\Omega_2, \calftwo, P,  \{\thtwot\}_{t \in \R})$
if  there exists $x_0 \in X$ such that for every  $c>0$,
$\omega_1 \in \Omega_1$ and $\omega_2 \in \Omega_2$,
$$
\lim_{t \to -\infty}
e^{c t} d (x_0, D(\thonet \omega_1, \thtwot \omega_2))
=0.
$$
\end{defn}

\begin{defn} 
\label{defepsneigh1}
A collection $\cald$ of some families 
of nonempty subsets of $X$
is said  to be   neighborhood closed if   for each
$D=\{D(\omega_1, \omega_2): 
\omega_1 \in \Omega_1, \omega_2 \in \Omega_2 \}
\in \cald$,   there exists a positive number
$\varepsilon$ depending on $D$ such that  the family
\be\label{defepsneigh2}
 \{ {B}(\omega_1, \omega_2) :
 {B}(\omega_1, \omega_2) \mbox{ is a  nonempty subset of }
 \caln_\varepsilon ( D (\omega_1, \omega_2) ),  \forall \
 \omega_1 \in \Omega_1,  \forall\  \omega_2
\in  \Omega_2\}
\ee
also belongs to $\cald$.
\end{defn}

\begin{defn}
\label{comporbit}
 Let $\cald$ be a collection of some families of
 nonempty  subsets of $X$. A mapping $\psi: \R \times \Omega_1 \times \Omega_2$
 $\to X$ is called a complete orbit of $\Phi$ if for every $\tau \in \R$, $t \ge 0$,
 $\omega_1 \in \Omega_1$ and $\omega_2 \in \Omega_2$,  the following holds:
\be
\label{comporbit1}
 \Phi (t, \theta_{1, \tau} \omega_1, \theta_{2, \tau} \omega_2,
  \psi (\tau, \omega_1, \omega_2) )
  = \psi (t + \tau, \omega_1, \omega_2 ).
\ee
 If, in  addition,    there exists $D=\{D(\omega_1, \omega_2): \omega_1 \in \Omega,
 \omega_2 \in \Omega_2 \}\in \cald$ such that
 $\psi(t, \omega_1, \omega_2)$ belongs to
 $D(\theta_{1,t} \omega_1, \theta_{2, t} \omega_2 )$
 for every  $t \in \R$, $\omega_1 \in \Omega_1$
 and $\omega_2 \in \Omega_2$, then $\psi$ is called a
 $\cald$-complete orbit of $\Phi$.
 \end{defn}

\begin{defn}
\label{defomlit}
Let $B=\{B(\omega_1, \omega_2): \omega_1 \in \Omega_1, \ \omega_2  \in \Omega_2\}$
be a family of nonempty subsets of $X$.
For every $\omega_1 \in \Omega_1$ and
$\omega_2 \in \Omega_2$,  let
\be\label{omegalimit}
\Omega (B, \omega_1, \omega_2)
= \bigcap_{\tau \ge 0}
\  \overline{ \bigcup_{t\ge \tau} \Phi(t, \theta_{1,-t} \omega_1,
 \theta_{2, -t} \omega_2, 
 B(\theta_{1,-t} \omega_1, \theta_{2,-t}\omega_2  ))}.
\ee
Then
the  family
 $\{\Omega (B, \omega_1, \omega_2): \omega_1
  \in \Omega_1, \omega_2 \in \Omega_2 \}$
 is called the $\Omega$-limit set of $B$
 and is denoted by $\Omega(B)$.
 \end{defn}

\begin{defn}
Let $\cald$ be a collection of some families
 of nonempty subsets of $X$ and
$K=\{K(\omega_1, \omega_2): \omega_1 
\in \Omega_1, \ \omega_2  \in \Omega_2\} \in \mathcal{D}$. Then
$K$  is called a  $\cald$-pullback
 absorbing
set for   $\Phi$   if
for all $\omega_1 \in \Omega_1$,
$\omega_2 \in \Omega_2 $
and  for every $B \in \cald$,
 there exists $T= T(B, \omega_1, \omega_2)>0$ such
that
\be
\label{abs1}
\Phi(t, \theta_{1,-t} \omega_1, \theta_{2, -t} \omega_2, 
B(\theta_{1,-t} \omega_1, \theta_{2,-t} \omega_2  )) 
 \subseteq  K(\omega_1, \omega_2)
\quad \mbox{for all} \ t \ge T.
\ee
If, in addition, for all $\omega_1 \in \Omega_1$ 
and $\omega_2 \in \Omega_2$,
   $K(\omega_1, \omega_2)$ is a closed 
   nonempty subset of $X$
   and $K$ is measurable with respect
    to the $P$-completion of $\calftwo$
   in $\Omega_2$,
 then we say $K$ is a  closed measurable
  $\cald$-pullback absorbing  set for $\Phi$.
\end{defn}

\begin{defn}
\label{asycomp}
 Let $\cald$ be a collection of  some families of  nonempty
 subsets of $X$.
 Then
$\Phi$ is said to be  $\cald$-pullback asymptotically
compact in $X$ if
for all $\omega_1 \in \Omega_1$ and
$\omega_2 \in \Omega_2$,    the sequence
\be
\label{asycomp1}
\{\Phi(t_n, \theta_{1, -t_n} \omega_1, \theta_{2, -t_n} \omega_2,
x_n)\}_{n=1}^\infty \mbox{  has a convergent  subsequence  in }   X
\ee
 whenever
  $t_n \to \infty$, and $ x_n\in   B(\theta_{1, -t_n}\omega_1,
  \theta_{2, -t_n} \omega_2 )$   with
$\{B(\omega_1, \omega_2): \omega_1 \in \Omega_1, \ \omega_2 \in \Omega_2
\}   \in \mathcal{D}$.
\end{defn}

\begin{defn}
\label{defatt}
 Let $\cald$ be a collection of some families of
 nonempty  subsets of $X$
 and
 $\cala = \{\cala (\omega_1, \omega_2): \omega_1 \in \Omega_1,
  \omega_2 \in \Omega_2 \} \in \cald $.
Then     $\cala$
is called a    $\cald$-pullback    attractor  for
  $\Phi$
if the following  conditions (i)-(iii) are  fulfilled:
\begin{itemize}
\item [(i)]   $\cala$ is measurable
with respect to the $P$-completion of $\calftwo$ in $\Omega_2$ and
 $\cala(\omega_1, \omega_2)$ is compact for all $\omega_1 \in \Omega_1$
and    $\omega_2 \in \Omega_2$.

\item[(ii)]   $\cala$  is invariant, that is,
for every $\omega_1 \in \Omega_1$ and
 $\omega_2 \in \Omega_2$,
$$ \Phi(t, \omega_1, \omega_2, \cala(\omega_1, \omega_2)   )
= \cala (\theta_{1,t} \omega_1, \theta_{2,t} \omega_2
), \ \  \forall \   t \ge 0.
$$

\item[(iii)]   $\cala  $
attracts  every  member   of   $\cald$,  that is, for every
 $B = \{B(\omega_1, \omega_2): \omega_1 \in \Omega_1, \omega_2 \in \Omega_2\}
 \in \cald$ and for every $\omega_1 \in \Omega_1$ and
 $\omega_2 \in \Omega_2$,
$$ \lim_{t \to  \infty} d (\Phi(t, \theta_{1,-t}\omega_1,
 \theta_{2,-t}\omega_2, B(\theta_{1,-t}\omega_1, 
 \theta_{2,-t}\omega_2) ) , \cala (\omega_1, \omega_2 ))=0.
$$
 \end{itemize}
 If, in addition, there exists $T>0$ such that
 $$
 \cala(\theta_{1, T} \omega_1, \omega_2) = \cala(\omega_1,    \omega_2 ),
 \quad \forall \  \omega_1 \in \Omega_1, \forall \
  \omega_2 \in \Omega_2,
 $$
 then we say $\cala$ is periodic with period $T$.
\end{defn}

  We borrow 
the  following result from \cite{wan4}  
regarding  the  
 existence and uniqueness of 
 $\cald$-pullback attractors.
 Similar results
for random    systems can be found in 
\cite{bat1,  cra2, fla1, schm1}.

\begin{prop}
\label{att}  
 Let $\cald$ be a   neighborhood closed  
 collection of some  families of   nonempty subsets of
$X$,  and $\Phi$  be a continuous   cocycle on $X$
over $(\Omega_1,  \{\thonet\}_{t \in \R})$
and
$(\Omega_2, \calftwo, P,  \{\thtwot\}_{t \in \R})$.
Then
$\Phi$ has a  $\cald$-pullback
attractor $\cala$  in $\cald$
if and only if
$\Phi$ is $\cald$-pullback asymptotically
compact in $X$ and $\Phi$ has a  closed
   measurable (w.r.t. the $P$-completion of $\calftwo$)
     $\cald$-pullback absorbing set
  $K$ in $\cald$.
  The $\cald$-pullback
attractor $\cala$   is unique   and is given  by,
for each $\omega_1  \in \Omega_1$   and
$\omega_2 \in \Omega_2$,
\be\label{attform1}
\cala (\omega_1, \omega_2)
=\Omega(K, \omega_1, \omega_2)
=\bigcup_{B \in \cald} \Omega(B, \omega_1, \omega_2)
\ee
\be\label{attform2}
 =\{\psi(0, \omega_1, \omega_2): \psi \mbox{ is a   }  \cald {\rm -}
 \mbox{complete orbit of } \Phi\} .
 \ee
  \end{prop}

For 
the periodicity of   $\cald$-pullback attractors,  we have the following 
result \cite{wan4}.

\begin{prop}
\label{periodatt}
 Let   $\cald$   be  a  neighborhood closed 
     collection of some  families of   nonempty subsets of
$X$.
Suppose  $\Phi$   is  a continuous  periodic   cocycle
with period $T>0$  on $X$
over $(\Omega_1,  \{\thonet\}_{t \in \R})$
and
$(\Omega_2, \calftwo, P,  \{\thtwot\}_{t \in \R})$.
 Suppose further that  $\Phi$  has a $\cald$-pullback attractor
 $\cala \in \cald$.  Then $\cala$ is periodic with period  $T$
 if and only if 
$\Phi$  has a  closed
   measurable (w.r.t. the $P$-completion of $\calftwo$)
     $\cald$-pullback absorbing set
  $K \in \cald$ with $K$ being periodic with period $T$. 
\end{prop}

Note that  a family  $K=\{K(\omega_1,
\omega_2): \omega_1 \in \Omega_1,
\omega_2 \in \Omega_2 \} \in \cald$
  is periodic with period $T$ 
if $K(\theta_{1, T} \omega_1, \omega_2)
=K( \omega_1, \omega_2)$
for all $ \omega_1 \in \Omega_1$
and  $\omega_2 \in \Omega_2$.
In the next section,  we  discuss the
continuity of pullback attractors when
a parameter varies.

\section{Upper Semicontinuity of Random Attractors}
\setcounter{equation}{0}

In this section, we discuss  the upper semicontinuity of
pullback   attractors  of a family of cocycles  
 on  a Banach space $X$.
 Suppose $\Lambda$  is   a  metric space.
 Given $\lambda\in \Lambda$,  let  $\Phi_\lambda$
 be  a continuous  cocycle  on $X$ over 
 $(\Omega_1,  \{\thonet\}_{t \in \R})$
and  $(\Omega_2, \calftwo, P,  \{\thtwot\}_{t \in \R})$.
Suppose   there exists $\lambda_0 \in \Lambda$
such  that for  every $t \in \R^+$,
$\omega_1 \in \Omega_1$,
 $\omega_2 \in \Omega_2$, 
$\lambda_n  \in \Lambda$ with
$\lambda_n \to \lambda_0$, and
$x_n$, $x \in X$ with $x_n \to x$, the following holds:
\be
 \label{sem1}
\lim_{n \to \infty} \Phi_{\lambda_n} (t, \omega_1, \omega_2,   x_n )
= \Phi_{\lambda_0} (t, \omega_1, \omega_2, x).
\ee
For   each 
$\lambda \in \Lambda$,
let $\cald_\lambda$ be a collection of 
 families of  nonempty subsets of $X$. 
Suppose there exists a map
$R_{\lambda_0}: \Omega_1 \times \Omega_2
\to \R$ such that  the family
\be
\label{sem1_1}
B=\{ 
B(\omega_1, \omega_2) =\{ x \in X: \| x \|_X 
\le R_{\lambda_0} (\omega_1, \omega_2) \}:
  \omega_1 \in \Omega_1, \omega_2 \in \Omega_2
\}
 \ \text{ belongs to } \  \cald_{\lambda_0}.
\ee
Suppose  further    that  for each 
$\lambda \in \Lambda$,   $\Phi_\lambda$ has a 
 $\cald_\lambda$-pullback  attractor
$\cala_\lambda  
 \in \cald_\lambda$ and a $\cald_\lambda$-pullback  absorbing
set $K_\lambda  
\in  \cald_\lambda   $  such that for all  $ \omega_1
\in \Omega_1$
and $ \omega_2 \in \Omega_2 $, 
\be 
\label{sem2}
\limsup_{\lambda  \to   \lambda_0}
\| K_\lambda (\omega_1, \omega_2) \|_X \le 
  R_{\lambda_0}  (\omega_1, \omega_2),
\ee
where  $R_{\lambda_0}$ is as in \eqref{sem1_1}
and $\| S \|_X =
\sup_{x \in  S} \| x \|_X$
for a subset $S$ of $X$.
We  finally  assume that   for every
$ \omega_1
\in \Omega_1$
and $ \omega_2 \in \Omega_2 $, 
\be
 \label{sem3}
\bigcup_{\lambda \in \Lambda } \cala_\lambda (\omega_1,
 \omega_2)
 \  \text{is precompact in } X.
\ee
 
 We now present     the 
 upper semicontinuity of
   $\cala_\lambda$  at
 $\lambda = \lambda_0$. 
  
\begin{thm}
\label{semcon}
Suppose \eqref{sem1}-\eqref{sem3} hold. Then for
every 
$\omega_1 \in \Omega_1$
and $\omega_2 \in \Omega_2$,
\be  \label{sem4}
 {\rm dist} (\cala_\lambda (\omega_1, \omega_2), 
 \cala_{\lambda_0} (\omega_1,  \omega_2 ))
\to 0,
\quad \text{as} \quad \lambda \to   \lambda_0.
\ee
\end{thm}

\begin{proof}
Suppose  \eqref{sem4} is false.  Then there
exist   a positive number $\eta$ 
and  a sequence $\lambda_n \to \lambda_0$ such that
for all $n \in \N$,  
\be\label{semp1}
{\rm dist} (\cala_{\lambda_n} (\omega_1, \omega_2), 
 \cala_{\lambda_0} (\omega_1,  \omega_2 ) )
\ge  2  \eta.
\ee
By \eqref{semp1} we find  that
there exists a    sequence $\{x_n\}_{n=1}^\infty$
with $x_n \in  \cala _{\lambda_n} (\omega_1, \omega_2)$ 
 such that  
 \be\label{semp2}
 {\rm dist}
 (x_n,  \cala_{\lambda_0} (\omega_1,  \omega_2 ))   \ge \eta
 \ \text{   for all }  \ n \in \N.
 \ee
 By  \eqref{sem3} we get  that  there  exists  $x_0 \in X$
such that, up to a subsequence, 
\be
 \label{semp3}
\lim_{n \to \infty} x_n =x_0.
\ee
We now  prove $x_0 \in \cala_{\lambda_0}
(\omega_1, \omega_2)$.  
Let 
 $\{t_m\}_{m=1}^\infty$
 be a sequence of  positive
 numbers such that   $t_m \to \infty$.
 Fix $m =1$. Then by the 
  invariance of $\cala_{\lambda_n}$
  for every $n \in \N$, 
there exists a sequence
$\{x_{1,n}\}_{n=1}^\infty$ with $x_{1,n} \in
\cala_{\lambda_n}(\theta_{-t_1} \omega_1,
 \theta_{-t_1} \omega_2 )$
such that   for all  $n \in \N$,
\be  \label{semp4}
x_n = \Phi_{\lambda_n} (t_1, \theta_{-t_1} \omega_1,
\theta_{-t_1} \omega_2, 
 x_{1,n}  ) .
\ee
Since $x_{1,n} \in
\cala_{\lambda_n}(\theta_{-t_1} \omega_1,
 \theta_{-t_1} \omega_2 )$
 for all $n \in \N$,  we get from 
  \eqref{sem3}  that  there exist
$z_1 \in X$ and a subsequence of $\{x_{1,n}\}_{n=1}^\infty$
(which we do not relabel)  such that
\be  \label{semp6}
\lim_{n \to \infty} x_{1,n} = z_1.
\ee
It follows    from  \eqref{sem1} and \eqref{semp6}    
\be  \label{semp7}
\lim_{n \to \infty}
\Phi_{\lambda_n} (t_1, \theta_{-t_1} \omega_1, 
\theta_{-t_1} \omega_2,
x_{1,n} )
= \Phi_{\lambda_0}  (t_1, \theta_{-t_1} \omega_1, 
\theta_{-t_1} \omega_2,   z_1  ).
\ee
By   \eqref{semp3}-\eqref{semp4}
and \eqref{semp7}
we obtain 
  \be  \label{semp8}
x_0 =  \Phi_{\lambda_0}  (t_1, \theta_{-t_1} \omega_1, 
\theta_{-t_1} \omega_2,   z_1  ).
\ee
Note that 
$  \cala_{\lambda_n}(\theta_{-t_1} \omega_1,
\theta_{-t_1} \omega_2)$
$ \subseteq { {K}}_{\lambda_n}(\theta_{-t_1} \omega_1,
\theta_{-t_1} \omega_2)$
and $   x_{1,n}  
\in   \cala_{\lambda_n}(\theta_{-t_1} \omega_1,
\theta_{-t_1} \omega_2)$    for all $n \in \N$.
Thus by \eqref{sem2} we  have
\be  \label{semp9}
\limsup_{n \to \infty} \| x_{1,n} \|_X
\le \limsup_{n \to \infty}
\| K_{\lambda_n} (\theta_{-t_1} \omega_1,
\theta_{-t_1} \omega_2)  \|_X
\le   R_{\lambda_0} (\theta_{-t_1} \omega_1,
\theta_{-t_1} \omega_2) .
\ee
By \eqref{semp6} and
\eqref{semp9} we  get
$\| z_1 \|_X \le   R_{\lambda_0} (\theta_{-t_1} \omega_1,
\theta_{-t_1} \omega_2)$.
By induction,   for every
 $m \ge 1$,   we find that 
there exists  $z_m \in X$ such that
for all   $ m \in \N$,
\be 
\label{semp10}
x_0 =  \Phi_{\lambda_0}  (t_m, \theta_{-t_m} \omega_1, 
\theta_{-t_m} \omega_2,   z_m  ),
\ee
and 
\be\label{semp11}
\| z_m \|_X \le 
 R_{\lambda_0} (\theta_{-t_m} \omega_1,
\theta_{-t_m} \omega_2).
\ee 
 By \eqref{sem1_1}    and the attraction
 property of $\cala_{\lambda_0}$ in $\cald_{\lambda_0}$,
  we obtain
 from \eqref{semp10}-\eqref{semp11}
 that 
 $$
 {\rm dist} (x_0,  \cala_{\lambda_0} (\omega_1, \omega_2))
 =  
 {\rm dist}  (\Phi_{\lambda_0}
 (t_m, \theta_{1, -t_m} \omega_1,
 \theta_{2, -t_m} \omega_2,  z_m ), \ 
  \cala_{\lambda_0} (\omega_1, \omega_2) )
  $$
  $$
  \le
 {\rm dist}  (\Phi_{\lambda_0}
 (t_m, \theta_{1, -t_m} \omega_1,
 \theta_{2, -t_m} \omega_2,  B (\theta_{1, -t_m} \omega_1,
\theta_{2, -t_m} \omega_2) ), \ 
  \cala_{\lambda_0} (\omega_1, \omega_2) )
  \to 0, \ \text{ as } m \to \infty.
  $$
  This shows   that $x_0 \in 
 \cala_{\lambda_0} (\omega_1, \omega_2)$
 since  $\cala_{\lambda_0} (\omega_1, \omega_2)$
 is compact. 
 Thus, by   \eqref{semp3}  we  get
$$ 
{\rm dist}(x_n,  \cala_{\lambda_0} (\omega_1, \omega_2) )
\le  {\rm dist}(x_n, x_0) \to 0, \ 
 \text{as }   n \to \infty.
 $$
 This is in contradiction with  \eqref{semp2}
 and hence proves \eqref{sem4}.
\end{proof}

Next,  we  consider  two special cases of 
Theorem \ref{semcon} where
the limiting cocycle $\Phi_{\lambda_0}$
is    an  autonomous  system or 
a deterministic
non-autonomous  system.
Both cases are interesting in their own right
and   deserve further discussions. 
Suppose $\Phi_{\lambda_0}: \R^+ \times 
\Omega_1 \times \Omega_2 \times X
\to X$  is a  continuous   cocycle.
If $\Phi_{\lambda_0}$ is constant in $\omega_2
\in \Omega_2$, then 
we can drop the dependence of
$\Phi_{\lambda_0}$ on $\omega_2$
and  call such 
  $\Phi_{\lambda_0}$
 a deterministic
 non-autonomous  cocycle.
In other words,  a mapping 
  $\Phi_{\lambda_0}:  \R^+ \times 
\Omega_1   \times X
\to X$ is a continuous  non-autonomous cocycle
on $X$ over $(\Omega_1, \{\thonet\}_{t \in \R})$
if  $\Phi_{\lambda_0}$ satisfies the following conditions:
for every $\omega_1 \in \Omega_1$  and 
$t, \tau \in \R^+$,  

  (i)   $\Phi(0, \omega_1, \cdot) $ is the identity on $X$;
  
  (ii)    $\Phi(t+\tau, \omega_1,  \cdot) =
 \Phi(t, \theta_{1,\tau} \omega_1,   \cdot) 
 \circ \Phi(\tau, \omega_1,  \cdot)$;
 
 (iii)    $\Phi(t, \omega_1, \cdot): X \to  X$ is continuous.
    
Let $\cald_{\lambda_0}$  be a collection of families of
nonempty subsets of $X$ given  by
$$
\cald_{\lambda_0}
=\{
 D=\{ D(\omega_1) \neq \emptyset:  D(\omega_1) \subseteq X, \omega_1
 \in \Omega_1
 \}
\}.
$$
A family $\cala_{\lambda_0} \in \cald_{\lambda_0}$
of compact subsets of $X$
is called  a $\cald_{\lambda_0}$-pullback attractor of $\Phi_{\lambda_0}$ if 
$\cala_{\lambda_0}  $ pullback 
attracts  every  member   of   $\cald_{\lambda_0}$
 and 
$  \Phi_{\lambda_0} (t, \omega_1, \cala(\omega_1)   )
= \cala (\theta_{1,t} \omega_1) $ 
for all $t \in \R^+$ and $\omega_1 \in \Omega_1$.
  As in the random case, 
we assume
there exists a map
${\overline{R}}_{\lambda_0}: \Omega_1  
\to \R$ such that  the family
\be
\label{sem1nd_1}
B=\{ 
B(\omega_1 ) =\{ x \in X: \| x \|_X 
\le {\overline{R}}_{\lambda_0} (\omega_1) \}:
  \omega_1 \in \Omega_1 
\}
 \ \text{ belongs to } \  \cald_{\lambda_0},
\ee
Let $K_\lambda \in \cald_\lambda$
be a $\cald_\lambda$-pullback  absorbing set   
of $\Phi_\lambda$  which satisfies,
  for all $\omega_1 \in \Omega_1$
and $\omega_2 \in \Omega_2$,
\be 
\label{sem2nd}
\limsup_{\lambda  \to   \lambda_0}
\| K_\lambda (\omega_1, \omega_2) \|_X \le 
  {\overline{R}}_{\lambda_0}  (\omega_1).
\ee
Note that  for a deterministic
 non-autonomous
 cocycle
  $\Phi_{\lambda_0}$, 
condition \eqref{sem1} becomes the following,
for each $\omega_1 \in \Omega_1$   and 
 $\omega_2 \in \Omega_2$, 
\be
 \label{sem1nd}
\lim_{n \to \infty} \Phi_{\lambda_n} (t, \omega_1, \omega_2,   x_n )
= \Phi_{\lambda_0} (t, \omega_1,  x),
\ee
where
  $t \in \R^+$,
$\lambda_n  \in \Lambda$ with
$\lambda_n \to \lambda_0$, and
$x_n$, $x \in X$ with $x_n \to x$.
 Replacing conditions \eqref{sem1}-\eqref{sem2} by
\eqref{sem1nd_1}-\eqref{sem1nd},  
 we      get the
following convergence result from Theorem \ref{semcon}
when
$\Phi_{\lambda_0}$ is a deterministic
non-autonomous 
cocycle.

\begin{thm}
\label{semcondn}
Suppose     \eqref{sem3}
and \eqref{sem1nd_1}-\eqref{sem1nd}  hold. Then for
every 
$\omega_1 \in \Omega_1$
and $\omega_2 \in \Omega_2$,
$$
 {\rm dist} (\cala_\lambda (\omega_1, \omega_2), 
 \cala_{\lambda_0} (\omega_1))
\to 0,
\quad \text{as} \quad \lambda \to   \lambda_0.
$$
\end{thm}

\begin{proof}
This theorem can be considered as  a 
special case of Theorem \ref{semcon}
where $\Phi_{\lambda_0}$, $\cald_{\lambda_0}$
and $\cala_{\lambda_0}$ are all constant
with respect to  $\omega_2 \in \Omega_2$.
On the other hand,  we can also prove
Theorem \ref{semcondn} directly by following
the proof of Theorem \ref{semcon}
with minor changes. The details are omitted.
\end{proof}

 We now consider  the case  where 
 $\Phi_{\lambda_0}$ is  an autonomous cocycle,
 i.e.,   $\Phi_{\lambda_0}$
 is constant in both  $\omega_1$  and $\omega_2$.
 In this case,   a nonempty compact subset 
 $\cala_{\lambda_0}$ of $X$ is called
 a global  attractor
 of $\Phi_{\lambda_0}$ if $\cala_{\lambda_0}$ is
 invariant and attracts every bounded set 
 uniformly.
 For  an autonomous cocycle $\Phi_{\lambda_0}$, 
 condition \eqref{sem2nd}
 can be replaced  by  the following: there exists 
 a positive number $C$
  such that   for all $\omega_1 \in \Omega_1$
  and $\omega_2 \in \Omega_2$,
 \be 
\label{sem2nd2}
\limsup_{\lambda  \to   \lambda_0}
\| K_\lambda (\omega_1, \omega_2) \|_X \le 
 C.
\ee
We also assume that  
for  every  $\omega_1 \in \Omega_1$   and 
 $\omega_2 \in \Omega_2$, 
\be
 \label{sem1nd2}
\lim_{n \to \infty} \Phi_{\lambda_n} (t, \omega_1, \omega_2,   x_n )
= \Phi_{\lambda_0} (t,   x),
\ee
where
  $t \in \R^+$,
$\lambda_n  \in \Lambda$ with
$\lambda_n \to \lambda_0$, and
$x_n$, $x \in X$ with $x_n \to x$.
 For an autonomous cocycle $\Phi_{\lambda_0}$,
 we have the following convergence result.

\begin{thm}
\label{semcondn2}
Suppose  \eqref{sem3}
and \eqref{sem2nd2}-\eqref{sem1nd2}  hold. Then for
every 
$\omega_1 \in \Omega_1$
and $\omega_2 \in \Omega_2$,
$$
 {\rm dist} (\cala_\lambda (\omega_1, \omega_2), 
 \cala_{\lambda_0}  )
\to 0,
\quad \text{as} \quad \lambda \to   \lambda_0.
$$
\end{thm}
 
 \begin{proof}
 The proof is similar to Theorem \ref{semcondn}
 and hence  omitted. 
 \end{proof}

  \section{Stochastic  Reaction-Diffusion Equations on $\R^n$}
\setcounter{equation}{0}

In the rest  of this paper,  we study the existence and 
upper semicontinuity of  tempered pullback 
attractors for  stochastic  Reaction-Diffusion  equations 
on $\R^n$  with  deterministic non-autonomous  terms
as well as multiplicative noise. 
  Given   $\tau \in\R$ and $t > \tau$,  consider      the
 following    equation
 defined for  $x \in \R^n$,
\be
  \label{rd1}
 {\frac {\partial u}{\partial t}}
  +  \lambda u - \Delta u 
     =  f(x, u)    +  g(t, x )   +  \alpha u \circ   {\frac {d\omega}{dt}},
 \ee
 with  initial  condition 
 \be\label{rd2}
 u(x, \tau) = u_\tau (x),   \quad x\in \R^n,
 \ee
where   $\lambda$  and $\alpha$
 are   positive constants,
$g \in L^2_{loc}(\R, \ltwo)$,  
$\omega$  is a  two-sided real-valued Wiener process on a probability space.
Note  that  equation \eqref{rd1} is understood in the sense
of Stratonovich integration. 
 The nonlinearity    $f$   is 
  a smooth
function  that satisfies, for some positive constants $\alpha_1$,
$\alpha_2$ and $\alpha_3$,
\be
\label{f1}
f(x, s) s \le -  \alpha_1 |s|^p + \psi_1(x), \quad \forall \ x \in \R^n, \ \ \forall \ s \in \R, 
\ee
\be
\label{f2}
|f(x, s) |   \le \alpha_2 |s|^{p-1} + \psi_2 (x),
\quad \forall \ x \in \R^n, \ \ \forall \ s \in \R, 
\ee
\be
\label{f3}
{\frac {\partial f}{\partial s}} (x, s)   \le \alpha_3,
\quad \forall \ x \in \R^n, \ \ \forall \ s \in \R, 
\ee
\be
\label{f4}
| {\frac {\partial f}{\partial x}} (x, s) | \le  \psi_3(x),
\quad \forall \ x \in \R^n, \ \ \forall \ s \in \R, 
\ee
where 
$\psi_1 \in L^1(R^n) \cap L^\infty(R^n)$  and $\psi_2, \psi_3 \in L^2(R^n)$.
 In this paper,  we will
 use the    probability space  
 $(\Omega, \calf , P)$, 
 where 
$\Omega = \{ \omega   \in C(\R, \R ):  \omega(0) =  0 \}$,
  $\calf$  is
 the Borel $\sigma$-algebra induced by the
compact-open topology of $\Omega$, and $P$
 is   the corresponding Wiener
measure on $(\Omega, \calf)$.   
 Define  a group  $\{\thtwot \}_{t \in \R}$  acting on  
 $(\Omega, \calf, P)$
   by
\be\label{rdshift1}
 \thtwot \omega (\cdot) = \omega (\cdot +t)
  - \omega (t), \quad  \omega \in \Omega, \ \ t \in \R .
\ee
Then $(\Omega, \calf, P, \{\thtwot\}_{t\in \R})$ is a  parametric 
dynamical  system.
 From \cite{arn1}  we know that      
  there exists a $\thtwot$-invariant set 
  $\tilde{\Omega}\subseteq \Omega$
of full $P$ measure  such that
for each $\omega \in \tilde{\Omega}$, 
\be\label{asome}
{\frac  {\omega (t)}{t}} \to 0 \quad \mbox {as } \ t \to \pm \infty.
\ee
In the sequel,  we    
  will write  $\tilde{\Omega}$ as 
$\Omega$  for   convenience. 
 Let $\Omega_1 =\R$ and for each
 $t \in \R$, define
 a map $\thonet: \R 
 \to \R$ by
 $\thonet (h) = h+ t$
 for all $h \in \R$. 
 We will   define  a continuous  cocycle  
   for  equation \eqref{rd1}
  over $ (\R, \{\thonet\}_{t \in \R} )$    and  $(\Omega, \mathcal{F},
P, \{\thtwot \}_{t\in \R})$.   To that end,  we need to
 transform   the stochastic    equation  into a 
  deterministic  non-autonomous  one.  
Given $\omega  \in \Omega$,  
let $z(t, \omega) = e^{-\alpha \omega (t)}$.
Then $z$ solves  the following stochastic equation
in the sense of Stratonovich integration:
\be
\label{zeq}
{\frac {dz}{dt}}
+ \alpha z \circ {\frac {d\omega}{dt}} =0.
\ee
Given $\tau \in \R$,  $t \ge \tau$, 
$\omega \in \Omega$  and $u_\tau
\in \ltwo$,   let  $u(t, \tau, \omega, v_\tau)$
satisfy \eqref{rd1} with initial condition 
$u_\tau$ at initial time $\tau$.
Then we introduce a  new variable
 $v(t, \tau, \omega, v_\tau)$  
  by
\be
\label{vu}
v (t, \tau, \omega, v_\tau)
=z(t, \omega) u (t, \tau, \omega, v_\tau)
\ \text{ with }  \ v_\tau =  z(\tau, \omega) u_\tau .
\ee
 By \eqref{rd1}-\eqref{rd2}  and \eqref{zeq} we   get  
 \be
  \label{v1}
 {\frac {\partial v}{\partial t}}
  +  \lambda v - \Delta  v
     = z(t, \omega)  f \left ( x,  z^{-1} (t, \omega) v 
     \right )    + 
     z(t, \omega)  g(t, x )    ,
 \ee
 with  initial  condition 
 \be\label{v2}
 v(x, \tau) = v_\tau (x),   \quad x\in \R^n. 
 \ee
 Note that     \eqref{v1}   is
 a deterministic  equation which is
  parametrized 
 by $\omega \in \Omega$. 
 Therefore, by a standard argument
 (see, e.g., \cite{bab1}),  one can show  that
   for every $\tau \in \R$, $\omega \in \Omega$
     and  $v_\tau \in \ltwo$,  problem 
      \eqref{v1}-\eqref{v2} has a unique solution
      $v \in C([\tau, \infty), \ltwo) \bigcap L^2_{loc}( 
      (\tau, \infty),  \hone )$.  In addition,  
      this solution is continuous 
      in  $v_\tau$  
      with   respect to the norm of $\ltwo$  and
         is $(\calf, \calb (\ltwo))$-measurable
      in $\omega \in \Omega$.    
      Based on this fact,  we can      define a cocycle
       $\Phi: \R^+ \times \R \times \Omega \times \ltwo$
$\to  \ltwo$ for      problem
\eqref{rd1}-\eqref{rd2} by using \eqref{vu}.
Given $t \in \R^+$,  $\tau \in \R$, $\omega \in \Omega$ 
 and $u_\tau \in \ltwo$,
let 
 \be \label{rdphi}
 \Phi (t, \tau,  \omega, u_\tau) = 
  u (t+\tau,  \tau, \theta_{2, -\tau} \omega, u_\tau) 
  = {\frac 1{z(t+\tau, \theta_{2, -\tau} \omega)}}
v(t+\tau, \tau,  \theta_{2, -\tau} \omega,  v_\tau),  
\ee
where $v_\tau = z(\tau, \theta_{2, -\tau} \omega) u_\tau $. 
By \eqref{rdphi}  one can check   that   for every
$t \ge 0$, $\tau \ge 0$, $r \in \R$  and 
$\omega \in \Omega$   
$$
\Phi (t + \tau, r, \omega,  \cdot )
=\Phi  (t , \tau + r,  \theta_{2,\tau}\omega,  \cdot)
\circ 
 \Phi (\tau, r, \omega,  \cdot) . 
$$
By   the measurability of $v$  in $\omega \in \Omega$
 and the continuity of $v$ in initial data
 $v_\tau \in \ltwo$, we see   that    $\Phi$ 
 as defined by \eqref{rdphi} 
is a continuous   cocycle on $\ltwo$ 
over $(\R,  \{\thonet\}_{t \in \R})$
and 
$(\Omega, \calf,
P, \{\thtwot \}_{t\in \R})$.
The rest of this paper is devoted to the 
existence and  convergence  of pullback attractors
for $\Phi$   in $\ltwo$.   
For this purpose,  we 
 need   to specify   a collection $\cald$ 
  of families of subsets  of $\ltwo$.

 As  usual, for   a  bounded nonempty  subset 
 $B$  of $\ltwo$,  we write 
   $  \| B\| = \sup\limits_{\psi \in B}
   \| \psi \|_{\ltwo }$. 
Suppose 
   $D =\{ D(\tau, \omega): \tau \in \R, \omega \in \Omega \}$  
    is   a  tempered family of
  bounded nonempty   subsets of $\ltwo $,  that is,  
  for every  $c>0$, $\tau \in \R$   and $\omega \in \Omega$, 
 \be
 \label{attd1}
 \lim_{t \to  - \infty} e^{  c  t} 
 \| D( \tau  +t, \thtwot \omega ) \|  =0. 
 \ee 
 From now on,  we use 
  $\cald$   to denote      the  collection of all  tempered families of
bounded nonempty  subsets of $\ltwo$, i.e.,
 \be
 \label{drd}
\cald = \{ 
   D =\{ D(\tau, \omega): \tau \in \R, \omega \in \Omega \}: \ 
 D  \ \mbox{satisfies} \  \eqref{attd1} \} .
\ee
Note that    
$\cald$  given  by \eqref{drd}
is  neighborhood closed. 
 For the external forcing $g$ we assume   that
  there exists
  $\delta \in [0,  \lambda)$
such that   for every $\tau \in \R$, 
 \be
 \label{gcon1}
 \int_{-\infty}^\tau e^{\delta s    } \| g (  s , \cdot )\|^2 d s
<  \infty.
 \ee
 Sometimes, the following tempered  condition is also needed 
 for $g$:  
 there   exists
  $\delta \in [0,  \lambda)$
such that  for every  $c>0$, 
 \be
 \label{gcon2}
 \lim_{t \to -\infty} e^{c t} 
 \int_{-\infty}^0   e^{\delta  s}
  \|g( s+t, \cdot ) \|^2     d s  =0.
  \ee
  Observe that condition \eqref{gcon2} is
  stronger than \eqref{gcon1}   for $g 
    \in L^2_{loc} (\R, \ltwo)$, and both 
   conditions   
do   not require that    $g$ 
is  bounded in $\ltwo$
at $ \pm \infty$.

\section{Uniform Estimates of Solutions}
\setcounter{equation}{0}

      In this section, we
 derive uniform estimates of  solutions  for
 problem \eqref{rd1}-\eqref{rd2}.
 These estimates are needed to prove 
    the existence  and continuity of 
$\cald$-pullback  attractors.

\begin{lem}
\label{lem1}
 Suppose  \eqref{f1}-\eqref{f4}  and \eqref{gcon1} hold.
Then for every $\tau \in \R$, $\omega \in \Omega$   
and $D=\{D(\tau, \omega)
: \tau \in \R,  \omega \in \Omega\}  \in \cald$,
 there exists  $T=T(\tau, \omega,  D)>0$ such that 
 for all $t \ge T$, the solution
 $v$ of  problem  \eqref{v1}-\eqref{v2} 
   with $\omega$ replaced by
 $\theta_{2, -\tau} \omega$  satisfies
$$
\| v(\tau, \tau -t,  \theta_{2, -\tau} \omega, v_{\tau -t}  ) \|^2 
 \le 
  c z^{-2} (-\tau, \omega)
   \int_{-\infty}^0
e^{\lambda s}   z^2(s,  \omega )  \left (1+   \| g(s+\tau, \cdot) \|^2 \right )   ds
$$
and
$$ 
\int_{\tau -t}^\tau e^{\lambda (s-\tau)} \left (
\|   v(s, \tau -t,\theta_{2, -\tau} \omega, v_{\tau -t} ) \|^2 _{\hone}
+  
 z^2(s, \theta_{2, -\tau} \omega )
\|    u(s, \tau -t, \theta_{2, -\tau} \omega, u_{\tau -t} )  \|^p_p
\right ) ds
$$
$$
\le  
  c z^{-2} (-\tau, \omega)
   \int_{-\infty}^0
e^{\lambda s}   z^2(s,  \omega )  \left (1+   \| g(s+\tau, \cdot) \|^2 \right )   ds,
$$
 where $v_{\tau -t}\in D(\tau -t, \theta_{2, -t} \omega)$ and
  $c$ is a  positive constant  independent of $\tau$, $\omega$, $D$ 
  and $\alpha$.
\end{lem}

\begin{proof}  
It follows   from   \eqref{v1} that
\be
\label{plem1_1}
{\frac 12} {\frac d{dt}} \|v\|^2 + \lambda \| v\|^2 + \| \nabla v \|^2 =
\ii  z(t, \omega) f(x, u )   v  dx  + z(t, \omega)(g, v).
 \ee
 By   \eqref{f1},  for the nonlinear term  in
    \eqref{plem1_1}   we have
\be\label{plem1_2}
\ii z(t, \omega)   f(x, u  )    v  dx
\le -\alpha_1 z^2(t, \omega) \| u \|^p_p
+ z^2(t, \omega) \ii \psi_1 dx .
\ee
Note that  the last   term  on the right-hand side of
\eqref{plem1_1} is bounded by
\be
\label{plem1_3}
z(t, \omega) |(g, v)|
\le {\frac 14} \lambda \| v\|^2 + {\frac 1\lambda} z^2(t, \omega)\|g \|^2.
\ee
It follows   from 
\eqref{plem1_1}-\eqref{plem1_3}   that
  \be
\label{plem1_4}
 {\frac d{dt}} \|v\|^2 + {\frac 32} \lambda \| v\|^2 
 +  2 \| \nabla v \|^2
 + 2 \alpha_1 z^2 (t, \omega) \| u  \|^p_p
 \le {\frac 2\lambda} z^2 (t, \omega) \| g \|^2
+  c_1 z^2(t, \omega).
\ee
First multiplying  \eqref{plem1_4} by $e^{\lambda t}$ and then
integrating over $(\tau -t, \tau)$ with $t \ge  0$,
we  obtain, for every $\omega \in \Omega$, 
$$
\| v(\tau, \tau -t, \omega, v_{\tau -t} ) \|^2
+  2\int_{\tau -t}^\tau e^{\lambda (s-\tau)}
\| \nabla v(s, \tau -t, \omega, v_{\tau -t} ) \|^2 ds
$$
$$
+ {\frac 12} \lambda
\int_{\tau -t}^\tau e^{\lambda (s-\tau)}
\|   v(s, \tau -t, \omega, v_{\tau -t} ) \|^2 ds
+ 2 \alpha_1 
 \int_{\tau -t}^\tau e^{\lambda (s-\tau)}z^2(s, \omega)
\|   u(s, \tau -t, \omega, u_{\tau -t}  )  \|^p_p ds
$$
$$
\le e^{-\lambda t} \| v_{\tau -t} \|^2
+{\frac 2\lambda} e^{-\lambda \tau} \int_{\tau -t}^\tau
e^{\lambda s}   z^2(s, \omega) \| g(s, \cdot ) \|^2   ds
+ c_1  \int_{\tau -t}^\tau e^{\lambda (s-\tau)} z^2(s, \omega) ds.
$$
Given $\tau \in \R$ and $\omega \in \Omega$, replacing 
 $\omega$ by $\theta_{2, -\tau}  \omega$ in the above, 
 we obtain for all $t \in \R^+$, 
$$
\| v(\tau, \tau -t, \theta_{2, -\tau} \omega, v_{\tau -t} ) \|^2
+ 2 \int_{\tau -t}^\tau e^{\lambda (s-\tau)}
\| \nabla v(s, \tau -t,\theta_{2, -\tau} \omega, v_{\tau -t} ) \|^2 ds
$$
$$
+ {\frac 12}\lambda
\int_{\tau -t}^\tau e^{\lambda (s-\tau)}
\|   v(s, \tau -t,\theta_{2, -\tau} \omega, v_{\tau -t} ) \|^2 
+ 2 \alpha_1 
 \int_{\tau -t}^\tau e^{\lambda (s-\tau)}
 z^2(s, \theta_{2, -\tau} \omega )
\|    u(s, \tau -t, \theta_{2, -\tau} \omega, u_{\tau -t} )  \|^p_p
$$
$$
\le e^{-\lambda t} \| v_{\tau -t} \|^2
+{\frac 2\lambda} e^{-\lambda \tau} \int_{\tau -t }^\tau
e^{\lambda s}  z^2(s, \theta_{2, -\tau} \omega ) 
\| g(s, \cdot) \|^2   ds
+   c_1  \int_{\tau -t}^\tau e^{\lambda (s-\tau)} 
 z^2(s, \theta_{2, -\tau} \omega ) ds
$$
\be\label{plem1_8}
\le e^{-\lambda t} \| v_{\tau -t} \|^2
+{\frac 2\lambda}z^{-2} (-\tau, \omega)
   \int_{-\infty}^0
e^{\lambda s}   z^2(s,  \omega )   \| g(s+\tau, \cdot) \|^2   ds
 +   c_1 z^{-2} (-\tau, \omega)
    \int_{-\infty }^0 e^{\lambda s}  z^2(s,  \omega ) ds.
\ee
By \eqref{asome} and \eqref{gcon1}   one
can check  that 
  the last two integrals on the right-hand side of 
\eqref{plem1_8} are well-defined.
On the other hand,
since
$v_{\tau -t} \in D(\tau -t, \theta_{2, -t} \omega)$
and $D \in \cald$,   we see  that 
  there exists $T=T(\tau, \omega, D)>0$   such that  for all
$t \ge T$,  
$$
e^{-\lambda t} \| v_{\tau -t} \|^2 
\le
e^{-\lambda t} \| D(\tau -t, \theta_{2, -t} \omega) \|^2 
\le 
z^{-2} (-\tau, \omega)
    \int_{-\infty }^0 e^{\lambda s}  z^2(s,  \omega ) ds,
    $$
which along with 
 \eqref{plem1_8} completes   the proof.  \end{proof}

As a  direct  consequence of Lemma \ref{lem1},
 we have the following
estimates.  
\begin{cor}
\label{cor2}
 Suppose  \eqref{f1}-\eqref{f4}  and \eqref{gcon1} hold.
Then for every $\tau \in \R$, $\omega \in \Omega$   
and $D=\{D(\tau, \omega)
: \tau \in \R,  \omega \in \Omega\}  \in \cald$,
 there exists  $T=T(\tau, \omega,  D) \ge 1$ such that 
 for all $t \ge T$, the solution
 $v$ of  problem  \eqref{v1}-\eqref{v2} 
   with $\omega$ replaced by
 $\theta_{2, -\tau} \omega$  satisfies
$$ 
\int_{\tau -1}^\tau   \left (
\|   v(s, \tau -t,\theta_{2, -\tau} \omega, v_{\tau -t} ) \|^2 _{\hone}
+  
 z^2(s, \theta_{2, -\tau} \omega )
\|    u(s, \tau -t, \theta_{2, -\tau} \omega, u_{\tau -t} )  \|^p_p
\right ) ds
$$
$$
\le 
  c z^{-2} (-\tau, \omega)
   \int_{-\infty}^0
e^{\lambda s}   z^2(s,  \omega )  \left (1+   \| g(s+\tau, \cdot) \|^2 \right )   ds,
 $$
 where $v_{\tau -t}\in D(\tau -t, \theta_{2, -t} \omega)$ and
  $c$ is a  positive constant  independent of $\tau$, $\omega$, $D$ 
  and $\alpha$.
\end{cor}

\begin{proof} 
This inequality  follows from  Lemma \ref{lem1}
and the fact 
 $e^{\lambda (s-\tau)} \ge e^{-\lambda}$
for   $\tau-1 \le s  \le   \tau$.   \end{proof}

The following estimates are needed when we derive
the convergence of pullback attractors.

\begin{lem}
\label{socon1}
Suppose  \eqref{f1}-\eqref{f4}  and \eqref{gcon1} hold.
Then for every $\tau \in \R$, $\omega \in \Omega$
 and $v_\tau \in \ltwo$, 
the solution
 $v$ of  problem  \eqref{v1}-\eqref{v2} 
    satisfies,   for
 all  $t  \ge \tau$,
 $$ 
 \|   v(t, \tau ,  \omega, v_{\tau } ) \|^2 +
\int_{\tau  }^t e^{\lambda (s-t)} \left (
\|   v(s, \tau ,  \omega, v_{\tau } ) \|^2 _{\hone}
+  
 z^2(s,  \omega )
\|    u(s, \tau ,  \omega, u_{\tau } )  \|^p_p
\right ) ds
$$
$$
\le \|v_\tau\|^2 + 
  c  
   \int_{\tau}^t
   z^2(s,  \omega )   \| g(s, \cdot) \|^2   ds
 +   c \int_{\tau}^t    z^2(s,  \omega ) ds,
 $$ 
  where $c$ is a  positive constant  independent of $\tau$, $\omega$
  and $\alpha$.
  \end{lem}
 
\begin{proof}
Multiplying  \eqref{plem1_4} by $e^{\lambda t}$ and then
integrating over $(\tau , t)$,
we  get,   
$$
\| v(t, \tau , \omega, v_{\tau } ) \|^2
+  2\int_{\tau }^t e^{\lambda (s-t)}
\| \nabla v(s, \tau , \omega, v_{\tau } ) \|^2 ds
$$
$$
+ {\frac 12} \lambda
\int_{\tau }^t e^{\lambda (s- t)}
\|   v(s, \tau , \omega, v_{\tau } ) \|^2 ds
+ 2 \alpha_1 
 \int_{\tau }^t e^{\lambda (s-t)}z^2(s, \omega)
\|   u(s, \tau , \omega, u_{\tau }  )  \|^p_p ds
$$
$$
\le e^{ \lambda (\tau -t ) } \| v_{\tau } \|^2
+{\frac 2\lambda}  \int_{\tau }^t
e^{\lambda (s-t)}   z^2(s, \omega) \| g(s, \cdot ) \|^2   ds
+ c_1  \int_{\tau }^t e^{\lambda (s-t)} z^2(s, \omega) ds
$$
 $$
\le   \| v_{\tau } \|^2
+{\frac 2\lambda}  \int_{\tau }^t
   z^2(s, \omega) \| g(s, \cdot ) \|^2   ds
+ c_1  \int_{\tau }^t   z^2(s, \omega) ds.
$$
This completes    the proof.  \end{proof}

  \begin{lem}
\label{lem3}
 Suppose  \eqref{f1}-\eqref{f4}  and \eqref{gcon1} hold.
Then for every $\tau \in \R$, $\omega \in \Omega$   
and $D=\{D(\tau, \omega)
: \tau \in \R,  \omega \in \Omega\}  \in \cald$,
 there exists  $T=T(\tau, \omega,  D) \ge 1$ such that 
 for all $t \ge T$, the solution
 $v$ of  problem  \eqref{v1}-\eqref{v2} 
  with $\omega$ replaced by
 $\theta_{2, -\tau} \omega$  satisfies
 $$
  \| \nabla v(\tau, \tau -t,\theta_{2, -\tau} \omega, v_{\tau -t} ) \|^2
\le 
 c z^{-2} (-\tau, \omega)
   \int_{-\infty}^0
e^{\lambda s}   z^2(s,  \omega ) 
\left (1 +   \| g(s+\tau, \cdot) \|^2  \right )  ds,
 $$
 where $v_{\tau -t}\in D(\tau -t, \theta_{2, -t} \omega)$ and
  $c$ is a  positive constant  independent of $\tau$, $\omega$, $D$ 
  and $\alpha$.
\end{lem}

\begin{proof}
Multiplying    \eqref{v1} by 
    $\Delta v$ and then integrating  over $\R^n$  we get  
\be
\label{plem3_1}
{\frac 12} {\frac d{dt}} \| \nabla v \|^2
+ \lambda \| \nabla  v \|^2
+ \| \Delta v \|^2
= - z(t, \omega) \ii f(x, u ) \Delta v dx
- z(t, \omega) (g, \Delta v ).
\ee
By \eqref{f3}-\eqref{f4},  the first term on the right-hand
side  of \eqref{plem3_1} satisfies
 $$
-  z(t, \omega)  \ii f(x,  u ) \;  \Delta v dx
=  z(t, \omega)   \ii {\frac {\partial f}{\partial x}} (x, u) \; \nabla  v dx
+   \ii {\frac {\partial f}{\partial u}} (x,u) \;   | \nabla v |^2 dx
$$
\be\label{plem3_2}
\le z(t, \omega) \| \psi_3 \| \| \nabla v \|
+\alpha_3  \| \nabla v \|^2
\le
\left ( \alpha_3 + {\frac 12} \right ) \| \nabla v \|^2
+{\frac 12} z^2(t, \omega) \| \psi_3 \|^2.
\ee
For  the last term on the right-hand side
of \eqref{plem3_1}  we have
\be
\label{plem3_3}
|- z(t, \omega)  (g, \Delta v)|  
\le {\frac 12} \| \Delta v \|^2
+  {\frac 12}  z^2 (t, \omega) \| g \|^2 .
\ee
By  \eqref{plem3_1}-\eqref{plem3_3}  we get 
$$
  {\frac d{dt}} \| \nabla v \|^2
+ 2 \lambda \| \nabla  v \|^2
+ \| \Delta  v \|^2 
 \le   (1 + 2 \alpha_3 ) \| \nabla v \|^2 
 + z^2(t, \omega) \| \psi_3 \|^2
 + z^2 (t, \omega) \| g \|^2.
 $$
Therefore we have   
\be\label{plem3_5}
  {\frac d{dt}} \| \nabla v \|^2
 \le   c_1  \| \nabla v \|^2 
 + z^2 (t, \omega) \| g \|^2
  + c_1 z^2(t, \omega) . 
\ee
 Given $t  \in \R^+$, $\tau \in \R$   and 
 $\omega \in \Omega$,  let 
  $s \in (\tau-1, \tau)$. 
  By   integrating  \eqref{plem3_5}   over 
 $(s, \tau)$     we  get
 $$
 \| \nabla v(\tau, \tau -t,   \omega, v_{\tau -t} ) \|^2
 \le
 \| \nabla v(s, \tau -t,  \omega, v_{\tau -t} ) \|^2 
 $$
 $$
 + c_1 \int_s^\tau \| \nabla v(\xi, \tau -t,  \omega, v_{\tau -t} ) \|^2  d \xi 
 + \int_s^\tau  z^2 (\xi, \omega)  \|g(\xi, \cdot )\|^2 d\xi
 + \int_s^\tau z^2 (\xi, \omega) d \xi.
 $$
 Integrating again  with respect to $s$ on  $(\tau-1, \tau)$, we  obtain
 $$
 \| \nabla v(\tau, \tau -t,   \omega, v_{\tau -t} ) \|^2
 \le
 (1+ c_1)  \int_{\tau -1} ^\tau \| \nabla v(\xi, \tau -t,  \omega, v_{\tau -t} ) \|^2  d \xi 
 + \int_{\tau -1}^\tau  z^2 (\xi, \omega)  \|g(\xi, \cdot )\|^2 d\xi
 + \int_{\tau -1} ^\tau z^2 (\xi, \omega) d \xi.
 $$
 Replacing $\omega$ by
 $ \theta_{2, -\tau} \omega$
 in the above,   we  find   
 $$
 \| \nabla v(\tau, \tau -t,     \theta_{2, -\tau} \omega, v_{\tau -t} ) \|^2
 \le
 (1+ c_1)  \int_{\tau -1} ^\tau \| \nabla v(s, \tau -t,  
 \theta_{2, -\tau} \omega , v_{\tau -t} ) \|^2  d s
 $$
 \be\label{plem3_9}
 + \int_{\tau -1}^\tau  z^2 (s,   \theta_{2, -\tau} \omega)  \|g(s, \cdot )\|^2 ds
 + \int_{\tau -1} ^\tau z^2 (s,   \theta_{2, -\tau} \omega) d s.
 \ee
 Note that  for  every   
 $\tau \in \R$   and $\omega \in \Omega$,
\be\label{plem3_10}
\int_{\tau -1} ^\tau z^2 (s,   \theta_{2, -\tau} \omega) d s
=
  z^{-2} ( -\tau, \omega) \int_{\tau -1}^\tau
  z^2(s-\tau, \omega)  ds
=   z^{-2} ( -\tau, \omega)   \int_{-1}^0    z^2(s , \omega)  ds
$$
$$
\le 
    e^\lambda z^{-2} ( -\tau, \omega)   \int_{-1}^0 e^{\lambda s}
   z^2(s , \omega)  ds
   \le
    e^\lambda z^{-2} ( -\tau, \omega)   \int_{- \infty}^0 e^{\lambda s}
   z^2(s , \omega)  ds. 
\ee
Similarly,  one can check 
\be\label{plem3_11}
  \int_{\tau -1}^\tau  z^2 (s,   \theta_{2, -\tau} \omega)  \|g(s, \cdot )\|^2 ds
 \le
    e^\lambda  z^{-2} (-\tau, \omega)
   \int_{-\infty}^0
e^{\lambda s}   z^2(s,  \omega )   \| g(s+\tau, \cdot) \|^2   ds.
\ee
By \eqref{plem3_9}-\eqref{plem3_11} we obtain,   for every
$\tau \in \R$, $\omega \in \Omega$     and $t \in \R^+$,
$$
 \| \nabla v(\tau, \tau -t,     \theta_{2, -\tau} \omega, v_{\tau -t} ) \|^2
 \le
 (1+ c_1)  \int_{\tau -1} ^\tau \| \nabla v(s, \tau -t,  
 \theta_{2, -\tau} \omega , v_{\tau -t} ) \|^2  d s
 $$ 
$$
   +
    e^\lambda z^{-2} ( -\tau, \omega)   \int_{- \infty}^0 e^{\lambda s}
   z^2(s , \omega)  ds
   +
   e^\lambda  z^{-2} (-\tau, \omega)
   \int_{-\infty}^0
e^{\lambda s}   z^2(s,  \omega )   \| g(s+\tau, \cdot) \|^2   ds,
$$
which along with   Corollary \ref{cor2} 
completes    the proof.  \end{proof}

In order to establish   the 
$\cald$-pullback asymptotic compactness
of problem \eqref{rd1}-\eqref{rd2}, we
need to derive  the    uniform estimates on the tails
of solutions  as given below.

  \begin{lem}
\label{lem4}
Suppose  \eqref{f1}-\eqref{f4}  and \eqref{gcon1} hold.
  Let  $\tau \in \R$, $\omega \in \Omega$   and $D=\{D(\tau, \omega)
: \tau \in \R,  \omega \in \Omega\}  \in \cald $.
Then  for every $\eta>0$,   there exist
    $T=T(\tau, \omega,  D, \eta) \ge 1$
 and $K=K(\tau, \omega, \eta) \ge 1$ such that for all $t \ge T$, the solution
 $v$ of equation \eqref{v1}  with $\omega$ replaced by
 $\theta_{2, -\tau} \omega$  satisfies
 \be\label{lem4a1}
 \int_{|x| \ge K}
 |v(\tau, \tau -t, \theta_{2, -\tau} \omega, v_{\tau -t} ) (x)|^2 dx \le \eta,
\ee
 where $v_{\tau -t}\in D(\tau -t, \theta_{2, -t} \omega)$.
\end{lem}

\begin{proof}
Let $\rho$ be a smooth function defined on $   \R^+$ such that
$0\le \rho(s) \le 1$ for all $s \in \R^+$, and
$$
\rho (s) = \left \{
\begin{array}{ll}
  0 & \quad \mbox{for} \ 0\le s \le 1; \\
 1 & \quad \mbox{for}  \  s \ge 2.
\end{array}
\right.
$$
By \eqref{v1}  we get 
 $$
 {\frac 12} {\frac d{dt}} \ii \rh |v|^2 dx
 +\lambda \ii \rh |v|^2 dx
 - \ii  \rh v    \Delta v dx
 $$
 \be
 \label{plem4_1}
=  \ii \rh z(t, \omega) f(x, u)  v dx
 + \ii  \rh z(t, \omega)  g   v dx.
 \ee
 Note that 
  \be\label{plem4_2}
     \ii  \rh v  \Delta v  dx
    \le 
    -  \int_{k\le |x| \le \sqrt{2} k} v \rhp {\frac {2x}{k^2}} \cdot \nabla v dx
    \le {\frac {c_1}k} (\| v \|^2 + \| \nabla v \|^2 ).
 \ee
   For the first    term on the right-hand side
   of \eqref{plem4_1},  by \eqref{f1}  we have      
\be\label{plem4_2}
  \ii \rh z(t, \omega) f(x, u)  v dx
  \le
 -    \alpha_1 z^2 (t, \omega)
\ii \rh |u|^p  dx +  z^2 (t, \omega)  \ii \rh  \psi_1 dx.
 \ee
 For the  last term on the right-hand side of \eqref{plem4_1}  we have
\be
\label{plem4_3}
|\ii  \rh z(t, \omega)  g   v dx|
\le 
{\frac 12} \lambda \ii \rh |v|^2 dx
+ {\frac 1{2\lambda} } z^2 (t, \omega) \ii \rh  g^2(t,x)   dx.
\ee
By  \eqref{plem4_1}-\eqref{plem4_3}  we get  
\be\label{plem4_5}
  {\frac d{dt}} \ii \rh |v|^2 dx
+   \lambda \ii \rh |v|^2 dx
\le {\frac {c_2} k}   \| v \|^2_\hone
+
c_2 z^2 (t, \omega)  \ii \rh
\left (   |\psi_1| 
+   g^2 \right ) dx.
\ee
Note that   $\psi_1 \in L^1(\R^n)$.
Therefore,  given $\eta>0$, there exists
$K_1 =K_1(\eta) \ge 1$ such that   for all  $k\ge K_1$,
\be\label{plem4_6}
c_2 \ii \rh |\psi_1|  dx
=c_2 \int_{|x| \ge k} |\psi_1| \rh dx
\le \eta .
\ee
   By  \eqref{plem4_5}-\eqref{plem4_6} 
   we find that 
there exists $K_2=K_2(\eta) \ge K_1$   such that   for all
$k \ge K_2$,
$$
  {\frac d{dt}} \ii \rh |v|^2 dx
+   \lambda \ii \rh |v|^2 dx
\le  \eta    \| v \|^2_\hone +\eta z^2 (t, \omega) 
+ c_2 z^2 (t, \omega)   \int_{|x| \ge k} 
     g^2(t, x)  dx.
$$
By   the Gronwall inequality,  we get for
each $t \in \R^+$, $\tau \in \R$,  $\omega \in \Omega$
and $k \ge K_2$, 
$$
\ii \rh |v(\tau, \tau -t, \omega, v_{\tau -t})|^2 dx
- e^{-\lambda t} \ii \rh |v_{\tau -t} |^2 dx
$$
$$
\le  \eta  
\int_{\tau -t}^\tau e^{\lambda (s-\tau) }
\|v(s, \tau -t, \omega, v_{\tau -t})\|^2_{H^1(\R^n)} ds
 +  \eta   \int_{\tau -t}^\tau e^{\lambda (s-\tau)}
 z^2(s, \omega) ds  
 $$
 \be\label{plem4_7}
+ c_2  \int_{\tau -t}^\tau
\int_{|x| \ge k}  e^{\lambda (s-\tau)}  z^2(s, \omega)  g^2(s, x) dx ds.
\ee
Since  \eqref{plem4_7} is valid for every $\omega \in \Omega$,  we can replace
  $\omega$ by $\theta_{2, -\tau} \omega$ to get,  for 
  $k \ge K_2$, 
  $$
\ii \rh |v(\tau, \tau -t, \theta_{2, -\tau}\omega, v_{\tau -t})|^2 dx
- e^{-\lambda t} \ii \rh |v_{\tau -t} |^2 dx
$$
$$
\le  \eta  
\int_{\tau -t}^\tau e^{\lambda (s-\tau) }
\|v(s, \tau -t, \theta_{2, -\tau}\omega, v_{\tau -t})\|^2_{H^1(\R^n)} ds
 +  \eta   \int_{\tau -t}^\tau e^{\lambda (s-\tau)}
 z^2(s, \theta_{2, -\tau}\omega) ds  
 $$
$$
+ c_2  \int_{\tau -t}^\tau
\int_{|x| \ge k}  e^{\lambda (s-\tau)}  z^2(s, \theta_{2, -\tau}\omega)  g^2(s, x) dx ds
$$
  $$
\le  \eta  
\int_{\tau -t}^\tau e^{\lambda (s-\tau) }
\|v(s, \tau -t, \theta_{2, -\tau}\omega, v_{\tau -t})\|^2_{H^1(\R^n)} ds
 +  \eta z^{-2} (-\tau, \omega)   \int_{-\infty}^0 e^{\lambda s }
 z^2(s, \omega) ds  
 $$
\be\label{plem4_10}
+ c_2 z^{-2} (-\tau, \omega)  \int_{-\infty}^0
\int_{|x| \ge k}  e^{\lambda s }  z^2(s,  \omega)  g^2(s +\tau, x) dx ds.
\ee
  Since $v_{\tau -t}  \in D(\tau -t, \theta_{2, -t} \omega)$
and $D \in \cald $ we see    that 
there exists $T_1=T_1(\tau, \omega, D, \eta) >0$
such that  for all $t \ge T_1$, 
\be\label{plem4_11}
e^{-\lambda t} \ii \rh |v_{\tau -t} (x)|^2 dx
\le
 e^{-\lambda t}     \|v_{\tau -t}  \|^2
 \le
 e^{-\lambda t}     \| D(\tau -t, \theta_{2, -t} \omega)  \|^2
       \le \eta.
 \ee
 On the other hand, by \eqref{asome} and \eqref{gcon1} one can check
 $$
  \int_{-\infty}^0
\ii   e^{\lambda s }  z^2(s,  \omega)  g^2(s +\tau, x) dx ds
=
 \int_{-\infty}^0
\ii   e^{(\lambda -\delta)  s }  z^2(s,  \omega)  e^{\delta s}  g^2(s +\tau, x) dx ds
< \infty,
 $$
 which implies  that there exists 
 $K_3 =K_3(\tau, \omega, \eta) \ge K_2$ such that
 for all $k \ge K_3$, 
 \be\label{plem4_12}
   c_2 z^{-2} (-\tau, \omega)  \int_{-\infty}^0
\int_{|x| \ge k}  e^{\lambda s }  z^2(s,  \omega)  g^2(s +\tau, x) dx ds
\le \eta.
\ee
It follows   from \eqref{plem4_10}-\eqref{plem4_12}
and Lemma \ref{lem1}
that there exists  $T_2=T_2(\tau, \omega, D, \eta)  \ge T_1$
such that  for all $t \ge T_2$  and $k \ge K_3$, 
$$
 \ii \rh |v(\tau, \tau -t, \theta_{2, -\tau}\omega, v_{\tau -t})|^2 dx
\le
  c \eta + c\eta  z^{-2} (-\tau, \omega)
    \int_{-\infty }^0 e^{\lambda s}  z^2(s,  \omega ) ds 
$$
\be\label{plem4_15}
 +     c\eta  z^{-2} (-\tau, \omega)
   \int_{-\infty}^0
e^{\lambda s}   z^2(s,  \omega )   \| g(s+\tau, \cdot) \|^2   ds.
 \ee
Since $ \rho(s) = 1$  for $s  \ge 2$ we have
    \be\label{plem4_20}
\int_{|x| \ge \sqrt{2} k} 
  |v(\tau, \tau -t, \theta_{2, -\tau}\omega, v_{\tau -t})|^2 dx
\le
\ii \rh |v(\tau, \tau -t, \theta_{2, -\tau}\omega, v_{\tau -t})|^2 dx.
\ee
By \eqref{plem4_15}-\eqref{plem4_20}  we get
\eqref{lem4a1}.  This completes    the proof. 
\end{proof}

\section{Tempered   Attractors  for 
Reaction-Diffusion Equations  }
 \setcounter{equation}{0}
 
 In   this  section,  we  prove    the existence of tempered 
 random  attractors
 for problem  \eqref{rd1}-\eqref{rd2}. 
 We first derive uniform estimates
 for the cocycle $\Phi$   as defined by
 \eqref{rdphi}.

\begin{lem}
\label{lemu1}
 Suppose  \eqref{f1}-\eqref{f4}  and \eqref{gcon1} hold.
Then for every $\tau \in \R$, $\omega \in \Omega$   
and $D=\{D(\tau, \omega)
: \tau \in \R,  \omega \in \Omega\}  \in \cald$,
 there exists  $T=T(\tau, \omega,  D) \ge 1$ such that 
 for all $t \ge T$,
 \be\label{lemu1a1}
\| u(\tau, \tau -t,  \theta_{2, -\tau} \omega, u_{\tau -t}  ) \|^2 _\hone
 \le   c    \int_{-\infty}^0
e^{\lambda s}   z^2(s,  \omega ) 
 \left (1 +   \| g(s+\tau, \cdot) \|^2  \right )  ds, 
\ee
 where $u_{\tau -t}\in D(\tau -t, \theta_{2, -t} \omega)$ and
  $c$ is a  positive constant  independent of $\tau$, $\omega$, $D$ 
  and $\alpha$.
\end{lem}

  \begin{proof}
Given $D=\{D(\tau, \omega): \tau \in \R, \omega \in \Omega\}
 \in \cald$,  define a new family  ${\bar{D}}$ for $D$  as follows:
 \be
 \label{plemu1_1}
 {\bar{D}} =  \left \{ {\bar{D}}(\tau, \omega): \ 
{\bar{D}}(\tau, \omega)
 = \{ v \in  \ltwo :  \| v \|
 \le  z^{-1} (- \tau,  \omega)  \| D(\tau, \omega) \| \},
  \tau \in \R, \omega \in \Omega  \right \}.
 \ee
 Since $D \in \cald$,   by   \eqref{asome} one can check 
  ${\bar{D}} $  also belongs to $D$, i.e., 
 $\bar{D}$ is tempered.
  Since $u_{\tau -t}\in D(\tau -t, \theta_{2, -t} \omega)$, 
  we find that  $v_{\tau -t} = z(\tau-t,  \theta_{2, -\tau} \omega)u_{\tau -t}$
  satisfies  
\be\label{plemu1_2}
 \| v_{\tau -t} \|
= \| z(\tau -t, \theta_{2, -\tau} \omega )  u_{\tau -t} \| 
\le   z^{-1} (t -\tau,  \theta_{2, -t} \omega)
 \ \| D(\tau -t, \theta_{2, -t} \omega)\|.
\ee
By \eqref{plemu1_1}-\eqref{plemu1_2}  we   see   that 
$v_{\tau -t} \in {\tilde{D}} (\tau -t, \theta_{2, -t} \omega )$.
Since ${\tilde{D}} \in  \cald $,  by  Lemmas \ref{lem1} and \ref{lem3},  
  there  exists  $T= T(\tau, \omega, D) \ge 1$   such that
for all $ t \ge T$,
$$
\| v(\tau  , \tau -t,  \theta_{2, -\tau} \omega, v_{\tau -t}  ) \|^2 _\hone
$$
\be\label{plemu1_5}
\le 
  c z^{-2} (-\tau, \omega)
   \int_{-\infty}^0
e^{\lambda s}   z^2(s,  \omega )   \| g(s+\tau, \cdot) \|^2   ds
 +   c  z^{-2} (-\tau, \omega)
    \int_{-\infty }^0 e^{\lambda s}  z^2(s,  \omega ) ds.
\ee
Note that \eqref{vu} implies
 $$
  v(\tau  , \tau -t,  \theta_{2, -\tau} \omega, v_{\tau -t}  )
  = z(\tau,   \theta_{2, -\tau} \omega ) \
     u (\tau  , \tau -t,  \theta_{2, -\tau} \omega, u_{\tau -t}  )
  = z^{-1} ( - \tau,    \omega ) \
    u  (\tau  , \tau -t,  \theta_{2, -\tau} \omega, u_{\tau -t}  ), 
   $$
 which along with \eqref{plemu1_5} completes   the proof.
\end{proof}

 By an argument similar to Lemma \ref{lemu1}, one can  establish  the
 following uniform estimates on the tails of solutions of problem
 \eqref{rd1}-\eqref{rd2}.
 
  \begin{lem}
\label{lemu2}
Suppose  \eqref{f1}-\eqref{f4}  and \eqref{gcon1} hold.
  Let  $\tau \in \R$, $\omega \in \Omega$   and $D=\{D(\tau, \omega)
: \tau \in \R,  \omega \in \Omega\}  \in \cald $.
Then  for every $\eta>0$,   there exist
    $T=T(\tau, \omega,  D, \eta) \ge 1$
 and $K=K(\tau, \omega, \eta) \ge 1$ such that for all $t \ge T$,  
 \be\label{lemu2a1}
 \int_{|x| \ge K}
 |u(\tau, \tau -t, \theta_{2, -\tau} \omega, u_{\tau -t} ) (x)|^2  dx \le \eta,
\ee
 where $u_{\tau -t}\in D(\tau -t, \theta_{2, -t} \omega)$.
\end{lem}

\begin{proof}
Following the proof of Lemma \ref{lemu1},  one  can obtain 
  \eqref{lemu2a1}  from   
  Lemma \ref{lem4}  directly. The details are omitted.
\end{proof}

\begin{lem}
\label{lemu3}
 Suppose  \eqref{f1}-\eqref{f4}  and \eqref{gcon2} hold. 
 Then the continuous cocycle $\Phi$ associated with
 problem \eqref{rd1}-\eqref{rd2} has a closed measurable
 $\cald$-pullback absorbing set
 $K \in \cald$ which is given by,  for each $
 \tau \in \R$   and $\omega \in \Omega$,
 \be
 \label{rdabs1}
  K(\tau, \omega) = \{ u\in  \ltwo : \|u \|^2
 \le  M(\tau, \omega)  \},
 \ee
 where  $M(\tau, \omega)$ 
 is  the number given by
 the right-hand side of \eqref{lemu1a1}.
\end{lem}

\begin{proof} 
 It is clear that for each $\tau \in \R$,  $M(\tau, \cdot): \Omega \to \R$
is   $(\calf, \calb(\R))$-measurable, and hence 
$K(\tau, \cdot): \Omega \to 2^H$ is a  measurable set-valued mapping.
By \eqref{asome} and \eqref{gcon2},
 after some calculations,  one can  check that
$K= \{ K(\tau, \omega): \tau \in \R,  \omega \in \Omega \}$ is tempered,
i.e., $K \in \cald$.
This along with Lemma \ref{lemu1} shows   that
$K$ is a closed measurable
 $\cald$-pullback absorbing set
 for $\Phi$ in $\cald$.
  \end{proof}

  We now prove   the $\cald$-pullback   asymptotic 
  compactness of  solutions of problem \eqref{rd1}-\eqref{rd2}.

 \begin{lem}
 \label{lemu4}
 Suppose  \eqref{f1}-\eqref{f4}  and \eqref{gcon2} hold. 
 Then the continuous cocycle $\Phi$ associated with
 problem \eqref{rd1}-\eqref{rd2}  
  is $\cald$-pullback
 asymptotically compact  in $\ltwo$,
that is, for  every $\tau \in \R$, $\omega \in \Omega$, 
 $D   \in \cald$,
  $t_n \to \infty$ and 
 $u_{0,n}  \in D(\tau -t_n, \theta_{2, -t_n} \omega )$,  the sequence
 $\Phi(t_n, \tau -t_n,  \theta_{2, -t_n} \omega,   u_{0,n}  ) $   has a
   convergent
subsequence in $\ltwo $.
\end{lem}

 \begin{proof}
 Given $K>0$,   let 
$ {Q}_K = \{ x \in \R^n:  |x|     \le K \}$
and  ${Q}^c_K = \R^n\setminus Q_K$. 
  By Lemma \ref{lemu2} we find   that
  for every $\varepsilon>0$,  $\tau \in \R$ and $\omega
  \in \Omega$,  there exist
  $K= K(\tau, \omega, \varepsilon) \ge 1$  
  and $N_1=N_1(\tau, \omega,  D, \varepsilon) \ge 1$
such that for  all $n \ge N_1$,
\be\label{plemu4_1}
\|   \Phi  (t_n  ,  \tau -t_n  ,  \theta_{2, -t_n} \omega,  
 u_{0, n}  )  \| _{L^2 ({Q}^c_{K} )
   }
\le \frac{\varepsilon}{2}.
\ee
On the other hand, 
By Lemma   \ref{lemu1}     there
exists $N_2=N_2(\tau, \omega, D, \varepsilon) \ge N_1$
such that for all $n \ge N_2 $,
$$
\|   \Phi  (t_n  ,  \tau -t_n  , 
\theta_{2, -t_n} \omega,  u_{0,n}    )  \| _{H^1( {Q}_{K } )
} \le  L(\tau, \omega),
$$
where $L(\tau, \omega)$ is a positive constant.
Then 
  the compact  embedding
$ H^1( {Q}_{K } )  \hookrightarrow  L^2 ( {Q}_{K } )$
together with  \eqref{plemu4_1}
implies 
   $  \{ \Phi  (t_n  ,  \tau -t_n  , \theta_{2, -t_n} \omega, 
    u_{0,n}    )\}_{n=1}^\infty $  has 
   a finite covering in
$ \ltwo   $ of balls of radii less than $\varepsilon$, and hence
$  \{ \Phi  (t_n  ,  \tau -t_n  , \theta_{2, -t_n} \omega, 
    u_{0,n}    )\}_{n=1}^\infty $ is precompact in $\ltwo$.
 \end{proof}

  We  are now  ready to present    the existence of 
  tempered pullback attractors for
  problem \eqref{rd1}-\eqref{rd2}.  
  
  \begin{thm}
\label{attrd1}
 Suppose  \eqref{f1}-\eqref{f4}  and \eqref{gcon2} hold. 
 Then the continuous cocycle $\Phi$ associated with
 problem \eqref{rd1}-\eqref{rd2}  
   has a unique $\cald$-pullback attractor $\cala
   =\{\cala(\tau, \omega):
      \tau \in \R, \ \omega \in \Omega \} \in \cald$
 in $\ltwo$.  Moreover,
 for each $\tau  \in \R$   and
$\omega \in \Omega$,
\be\label{attrd1_1}
\cala (\tau, \omega)
=\Omega(K, \tau, \omega)
=\bigcup_{B \in \cald} \Omega(B, \tau, \omega)
\ee
\be\label{attrd1_2}
 =\{\psi(0, \tau, \omega): \psi \mbox{ is a   }  \cald {\rm -}
 \mbox{complete orbit of } \Phi\} .
\ee
\end{thm}

\begin{proof}
This is an immediate consequence of
 Lemmas \ref{lemu3}, \ref{lemu4}
 and    Proposition \ref{att}.
\end{proof}

    We now  consider   the
    periodicity of the $\cald$-pullback
    attractor obtained in Theorem \ref{attrd1}. 
    Suppose   $g: \R \to  \ltwo$
  is    $T$-periodic  for  
some  $T>0$  and 
   $ g \in L^2  ((0,T), \ltwo)$.
   In this case,      $g$  satisfies  
  \eqref{gcon2}
  for  any $\delta>0$
  and  the cocycle $\Phi$ is $T$-periodic.
  Indeed,  
for every 
 $t \in \R^+$, $\tau \in \R$ and $\omega \in \Omega$,
   we  have     
$$
\Phi (t, \tau +T, \omega, \cdot )
= u(t+ \tau +T, \tau +T,  \theta_{2, -\tau -T} \omega, \cdot)
=u(t +\tau, \tau, \theta_{2, -\tau} \omega, \cdot  )
= \Phi (t, \tau,  \omega,  \cdot ). $$
On the other hand,
by   \eqref{lemu1a1} we find that
the $\cald$-pullback absorbing set
$K$ of $\Phi$  given by \eqref{rdabs1}
is   also $T$-periodic,
i.e.,
$K(\tau +T, \omega) = K(\tau, \omega)$
for all $\tau \in \R$  and
$\omega \in \Omega$. 
Then 
by   Proposition \ref{periodatt}, 
we obtain  the 
     periodicity of the 
$\cald$-pullback attractor 
which is  stated below.

  \begin{thm}
\label{rdatt2}
Let   \eqref{f1}-\eqref{f4}  hold. 
Suppose  $g: \R \to  \ltwo$  is    periodic  
with period $T>0$ and 
   $ g \in  L^2 ((0,T), \ltwo)$.  
 Then the continuous cocycle $\Phi$ associated with
 problem \eqref{rd1}-\eqref{rd2}  
   has a unique $T$-periodic   $\cald$-pullback attractor $\cala \in \cald$
 in $ \ltwo$.
\end{thm}

\section{Convergence  of   Attractors for Reaction-Diffusion Equations}
\setcounter{equation}{0}

In this section, we prove the upper semicontinuity
of  tempered  attractors  of 
problem \eqref{rd1}-\eqref{rd2}
when $\alpha \to 0$.
 To indicate  the
 dependence of solutions on $\alpha$,
  we will  write the solution of problem \eqref{rd1}-\eqref{rd2}
as $u_\alpha$, and the corresponding cocycle as $\Phi_\alpha$.
Similarly,  we write   the solution of problem
\eqref{v1}-\eqref{v2} as $v_\alpha$, that is, $v_\alpha$ satisfies
\be
  \label{va1}
 {\frac {\partial v_\alpha}{\partial t}}
  +  \lambda v_\alpha - \Delta  v_\alpha
     =  e^{-\alpha \omega (t)}  f \left ( x,     e^{\alpha \omega (t)} v_\alpha 
     \right )    + 
     e^{-\alpha \omega (t)}  g(t, x )    ,
 \ee
 with  initial  condition 
 \be\label{va2}
 v_\alpha( \tau, x ) = v_{\alpha, \tau}  (x),   \quad x\in \R^n. 
 \ee
When $\alpha =0$,  the stochastic  problem 
\eqref{rd1}-\eqref{rd2}   reduces
to  a   deterministic one:
\be
  \label{drd1}
 {\frac {\partial u}{\partial t}}
  +  \lambda u - \Delta u 
     =  f(x, u)    +  g(t, x )   ,
 \ee
  with  initial  condition 
 \be\label{drd2}
 u( \tau, x) = u_\tau (x),   \quad x\in \R^n.
 \ee 
 Throughout this section,  we assume $\alpha \in [0, 1]$.
It follows   from Theorem \ref{attrd1} that,  for
every  positive $\alpha$, $\Phi_\alpha$ has
a $\cald$-pullback attractor $\cala_\alpha \in \cald$.
  Let $\Phi_0$ be the continuous deterministic   cocycle 
associated with problem \eqref{drd1}-\eqref{drd2}
on $\ltwo$ over $(\R,  \{\thonet\}_{t \in \R})$.
Denote by $\cald_0$ the collection of tempered  families of
deterministic  nonempty  subsets of $\ltwo$, i.e.,
$$
\cald_0 = \{ D= \{ D(\tau)  \subseteq \ltwo:  \tau \in \R \}:\ 
\lim_{t \to  -\infty} e^{ct} \| D(\tau +t)\| =0, \  \forall  \tau \in \R,  \
\forall  \ c>0 \}.
$$
Under conditions \eqref{f1}-\eqref{f4} and \eqref{gcon2}, 
it  can be proved   that  
$\Phi_0$  has a unique $\cald_0$-pullback attractor
$\cala_0 = \{ \cala(\tau): \tau \in \R\} \in \cald_0$
in $\ltwo$
(see  \cite{wan6}).
Notice   that
the existence of   $\cald_0$-pullback attractors
for $\Phi_0$ is also implied by
Theorem \ref{attrd1}  as a special   case.

 Given $0< \alpha \le 1$, 
 let $K_\alpha$ be the $\cald$-pullback
 absorbing set  of $\Phi_\alpha$
 as defined by \eqref{rdabs1}, i.e.,
 \be
 \label{upabs1} 
K_\alpha = \left \{
K_\alpha  (\tau, \omega) =  \{ u \in \ltwo: \ \| u \|
\le M_\alpha (\tau, \omega ) \}: \ 
\tau \in \R, \ \omega \in \Omega
\right \},
\ee
where  $ M_\alpha (\tau, \omega )$ is given by 
\be\label{upabs2}
M_\alpha (\tau, \omega ) = 
    c \left (\int_{-\infty}^0 e^{\lambda s} e^{-2\alpha \omega (s)}
\left ( 1 + \| g(s+\tau, \cdot) \|^2 
\right ) ds  
\right )^{\frac 12} .
\ee
Note that the positive constant $c$ 
in \eqref{upabs2} 
 is  independent of $\tau$,
$\omega$ and $\alpha$.
Similarly,  let $K_0$ be a family of subsets of $\ltwo$ given by
 \be
 \label{upabs3} 
K_0 = \left \{
K_0  (\tau) =  \{ u \in \ltwo: \ \| u \|
\le M_0 (\tau ) \}: \ 
\tau \in \R 
\right \},
\ee
where  $ M_0 (\tau  )$ is  the constant:
\be\label{upabs3a1}
M_0 (\tau ) = 
    c \left (\int_{-\infty}^0 e^{\lambda s}  
\left ( 1 + \| g(s+\tau, \cdot) \|^2 
\right ) ds  
\right )^{\frac 12} .
\ee
It is evident that Lemma \ref{lemu1}
implies that 
$K_0$ is a $\cald_0$-pullback absorbing set
of $\Phi_0$ in $\ltwo$.
Given $\tau \in \R$   and $\omega \in \Omega$, denote by
\be\label{upabs4}
 B  (\tau, \omega) =  \{ u \in \ltwo: \ \| u \|
\le R (\tau, \omega ) \} ,
\ee
where  $ R (\tau, \omega )$ is given by 
 \be\label{upabs5}
R (\tau, \omega ) = 
    c \left (\int_{-\infty}^0 e^{\lambda s} e^{2  | \omega (s) |}
\left ( 1 + \| g(s+\tau, \cdot) \|^2 
\right ) ds  
\right )^{\frac 12} .
\ee
By \eqref{upabs1}-\eqref{upabs2}
and \eqref{upabs4}-\eqref{upabs5}
we have 
$K_\alpha (\tau , \omega)
\subseteq B(\tau, \omega)$
  for all $\alpha \in (0,1]$,
$\tau \in \R$   and
$\omega \in \Omega$. 
This implies   that
for every $\tau \in \R$   and
$\omega \in \Omega$,
\be\label{upabs5a1}
\bigcup_{0<\alpha \le 1}
\cala_\alpha (\tau, \omega)
\subseteq
\bigcup_{0<\alpha \le 1}
 K_\alpha (\tau, \omega)
\subseteq B(\tau, \omega).
\ee
By Lemma \ref{lemu1} we find that, 
for every  $0< \alpha \le 1$,
$\tau \in \R$
and  $\omega \in \Omega$, there exists $T>0$
  such that for all $t \ge T$,
\be \label{upabs6}
\| \Phi_\alpha (t,  \tau -t,  \theta_{2,-t} \omega,
\cala_\alpha (\tau -t, \theta_{2,-t}\omega ) ) \|_{\hone}
\le  R(\tau, \omega),
\ee
where $R(\tau, \omega)$ is given by  \eqref{upabs5}.
 By  \eqref{upabs6} and 
 the  invariance of  $\cala_\alpha$,  
 we get   that,  for  every 
 $\tau \in \R$
and  $\omega \in \Omega$, 
\be
 \label{upabs7}
\| u \|_{\hone} \le  R(\tau, \omega)
\quad \text{ for all }  \  u \in  
\cala_\alpha (\tau, \omega) \
 \text { with } \  0< \alpha \le 1.
\ee
We  will use    \eqref{upabs7} 
to prove
  the precompactness of the union
 of $\cala_\alpha$   
  in $\ltwo$  for  $0<\alpha \le 1$.

In   the sequel,  we further 
  assume  the nonlinear function
$f$ satisfies, for all $x \in \R^n$ and $s \in \R$,
\be
 \label{f5}
|{\frac {\partial f}{\partial s}} (x,s)|
\le c  |s|^{p-2} + \psi_4 (x),
\ee
where $c$  is a positive constant,
 $\psi_4 \in L^\infty(\R^n)$ if $p=2$, and
$\psi_4 \in L^{\frac p{p-2}} (\R^n)$ if $p>2$.

We will investigate the convergence of $\cala_\alpha$ as
$\alpha \to 0$. To that   end,  we first  derive 
the convergence of solutions of problem
\eqref{rd1}-\eqref{rd2}   as $\alpha \to 0$.

\begin{lem}
 \label{atup1}
 Suppose 
\eqref{f1}-\eqref{f4} and \eqref{f5} hold.
Let $v_\alpha$ and $u$ be the
solutions of  \eqref{va1}  and \eqref{drd1}
with initial conditions $v_{\alpha, \tau}$ and $u_\tau$, respectively.
Then,  for every $\tau \in \R$,
$\omega \in \Omega$,
$T>0$   and $\varepsilon \in [0, 1]$,
there exists a positive number
 $\alpha_0 = \alpha_0(\tau, \omega, T, \varepsilon)$
 such that for all $\alpha \le \alpha_0$  and $t \in [\tau,
 \tau +T]$,   
 $$
  \| v_\alpha (t, \tau, \omega, v_{\alpha, \tau}) - u(t,\tau, \omega)  \|^2
  \le c  e^{c  (t -\tau)} \|  v_{\alpha, \tau} -  u_\tau  \|^2
  $$
    \be\label{atup1a1}
  + c  \varepsilon e^{c  (t -\tau)}
  \left  ( T +
\|u_\tau\|^2 + \|u_{\alpha, \tau} \|^2 + 
   \int_{\tau}^t   \| g(s, \cdot) \|^2   ds
   \right ), 
 \ee
where $c$ is a positive   constant independent of  $\tau$, $\omega$, $\varepsilon$
and $\alpha$.  
\end{lem}

\begin{proof}
Let  $\kappa =  v_\alpha -u$.  Then by \eqref{va1}  and 
\eqref{drd1}  we get
\be\label{patup1_1}
{\frac {\partial \kappa}{\partial t}}  +\lambda \kappa-\Delta \kappa
= e^{-\alpha \omega (t)}  f \left ( x,     e^{\alpha \omega (t)} v_\alpha 
     \right )  -f(x,u)   + 
     \left (e^{-\alpha \omega (t)} -1 \right )   g(t, x ) .
\ee
Given $\tau \in \R$,
$\omega \in \Omega$,
$T>0$   and $\varepsilon \in [0, 1)$, since 
$\omega$ is continuous on $\R$,   we find   that
there  exists 
 $\alpha_0 = \alpha_0(\tau, \omega, T, \varepsilon)>0$
 such that for all $\alpha  \in [0,  \alpha_0]$  and $t \in [\tau,
 \tau +T]$,    
 \be
 \label{patup1_2}
 |e^{\alpha \omega (t)} -1| +  | e^{-\alpha \omega (t)} -1|
 < \varepsilon.
 \ee
Multiplying \eqref{patup1_1}
 by $\kappa$ and then integrating over $\R^n$, we
obtain
$$
 {\frac 12} {\frac d{dt}} \| \kappa\|^2
+\lambda \| \kappa\|^2 + \| \nabla \kappa \|^2 
$$
 \be
 \label{patup1_3}
=\ii 
 \left ( e^{-\alpha \omega (t)}  f \left ( x,     e^{\alpha \omega (t)} v_\alpha 
     \right )  -f(x,u) \right )  \kappa dx
     +   \left (e^{-\alpha \omega (t)} -1 \right )   \ii g(t, x ) \kappa dx.
\ee
By \eqref{f2}-\eqref{f3}  and \eqref{f5} we have the following estimates
on the first term on the right-hand side of \eqref{patup1_1}
$$
\ii 
 \left ( e^{-\alpha \omega (t)}  f  ( x,     e^{\alpha \omega (t)} v_\alpha  )
   -f(x,u) \right )  \kappa dx
= \ii 
  e^{-\alpha \omega (t)} \left ( f  ( x,     e^{\alpha \omega (t)} v_\alpha ) 
      -    f ( x,     e^{\alpha \omega (t)} u ) 
       \right )  \kappa dx
     $$
     $$
+  \ii  \left (
  e^{-\alpha \omega (t)}  f  ( x,     e^{\alpha \omega (t)} u ) 
      -    f ( x,     e^{\alpha \omega (t)} u )  \right )
       \kappa dx
       +  \ii   \left (  f  ( x,     e^{\alpha \omega (t)} u ) 
      -    f ( x,    u )  \right )
       \kappa dx
     $$
     $$
     = \ii \kappa^2  {\frac {\partial  f} {\partial s}} (x,s)   dx
+     \left (
  e^{-\alpha \omega (t)} -1 \right )
   \ii f  ( x,     e^{\alpha \omega (t)} u )   \kappa dx
       +  \left (
  e^{\alpha \omega (t)} -1 \right )  \ii   
       \kappa u   {\frac {\partial  f} {\partial s}} (x,s)   dx
     $$
     $$
     \le \alpha_3 \| \kappa \|^2
     +  | e^{-\alpha \omega (t)} -1  |
     \ii \left (
      \alpha_2 e^{\alpha (p-1) \omega (t)} |u|^{p-1} |\kappa|
      +\psi_2 | \kappa| 
     \right ) dx
     $$
     $$
     +  | e^{\alpha \omega (t)} -1  |  \ii  
     \left (
      c  (1+ e^{\alpha (p-2) \omega (t)}) |u|^{p-1} |\kappa|
      + \psi_4 |u| |\kappa| 
     \right )dx
     $$
      $$
     \le \alpha_3 \| \kappa \|^2
     +  c_1 | e^{-\alpha \omega (t)} -1  |
     \ii \left (
      e^{\alpha (p-1) \omega (t)} (  |u|^{p} +  |v_\alpha|^p )
      + |\psi_2|^2 +  | \kappa| ^2
     \right ) dx
     $$
     \be\label{patup1_5}
     + c_1  | e^{\alpha \omega (t)} -1  |  \ii  
     \left (
       (1+ e^{\alpha (p-2) \omega (t)}) (|u|^{p} +  |v_\alpha|^p )
      + |\psi_4  |^{{\frac p{p-2}}}
     \right )dx.
     \ee
By \eqref{patup1_2}  and \eqref{patup1_5}  we get
for all $\alpha \in [0, \alpha_0]$   and $t \in [\tau, \tau +T]$,
  \be\label{patup1_6}
\ii \left ( e^{-\alpha \omega (t)}  f  ( x,     e^{\alpha \omega (t)} v_\alpha  )
   -f(x,u) \right )  \kappa dx
   \le c_2 \| \kappa\|^2 +  c_2 \varepsilon +
   c_2\varepsilon \ii ( |u|^p + |v_\alpha|^p  ) dx.
 \ee
   For the last term on the right-hand side of
   \eqref{patup1_3}, by \eqref{patup1_2}  we have,
   for all $\alpha \in [0, \alpha_0]$   and $t \in [\tau, \tau +T]$,
   \be
   \label{patup1_7}
    \left (e^{-\alpha \omega (t)} -1 \right )   \ii g(t, x ) \kappa dx
    \le \varepsilon \| \kappa \|^2
    + \varepsilon \| g(t, \cdot )\|^2.
\ee
It follows    from  \eqref{patup1_3}
and \eqref{patup1_6}-\eqref{patup1_7} that 
 for all $\alpha \in [0, \alpha_0]$   and $t \in [\tau, \tau +T]$,
\be
 \label{patup1_9}
  {\frac d{dt}} \| \kappa\|^2
 \le  c_3 \| \kappa \|^2 
 +  c_4 \varepsilon  \left ( 
  \|u \|^p_p + \|v_\alpha \|^p_p   + \| g(t, \cdot) \|^2
  \right ).
 \ee
 By \eqref{patup1_9} we get,
  for all $\alpha \in [0, \alpha_0]$   and $t \in [\tau, \tau +T]$,
  \be\label{patup1_20}
  \| \kappa (t) \|^2
  \le e^{c_3 (t -\tau)} \| \kappa (\tau )\|^2
  + c_4 \varepsilon e^{c_3 (t -\tau)}
  \int_\tau^t
  \left (
  \| v_\alpha (s, \tau, \omega, v_{\alpha, \tau})\|^p_p
  +  \| u  (s, \tau,  u_\tau )\|^p_p + \| g(s,\cdot)\|^2
  \right )ds.
 \ee
 On  the other hand, 
  by  \eqref{patup1_2}  and
   Lemma \ref{socon1}  we obtain,  
   for all $\alpha \in [0, \alpha_0]$  and 
       $t \in [\tau, \tau +T]$,   
 $$ 
 \|   v_\alpha (t, \tau ,  \omega, v_{\alpha, \tau } ) \|^2 +
\int_{\tau  }^t   
 z^2(s,  \omega )
\|    u_\alpha (s, \tau ,  \omega, u_{\tau } )  \|^p_p ds
$$
$$
\le e^{\lambda (t-\tau)}
\left  ( 
 e^{-2\alpha \omega (\tau)} \|u_{\alpha, \tau} \|^2 + 
  c  
   \int_{\tau}^t e^{-2\alpha \omega (s)} (1+  \| g(s, \cdot) \|^2 )  ds
 \right )
 $$
 \be\label{patup1_21}
 \le c  e^{\lambda (t-\tau)}
\left  (  T+  \|u_{\alpha, \tau} \|^2 + 
   \int_{\tau}^t     \| g(s, \cdot) \|^2    ds
 \right ) .
 \ee
 Note that 
   Lemma \ref{socon1} is also valid for $\alpha =0$.
   Therefore   we  have, 
   for all   $t \in [\tau, \tau +T]$,   
  \be\label{patup1_22}
\int_{\tau  }^t   
\|    u(s, \tau ,   u_{\tau } )  \|^p_p ds
\le c e^{\lambda (t-\tau)}
\left  ( T +
\|u_\tau\|^2 + 
   \int_{\tau}^t   \| g(s, \cdot) \|^2   ds
   \right ) .
 \ee
 By \eqref{patup1_2} we find that, 
  for all $\alpha \in [0, \alpha_0]$   and $s \in [\tau, \tau +T]$,
  $$
  \| v_\alpha (s, \tau, \omega, v_{\alpha, \tau})\|^p_p
  = z^p(s, \omega)   \| u_\alpha (s, \tau, \omega, u_{\alpha, \tau})\|^p_p
  \le c z^2(s, \omega)   \| u_\alpha (s, \tau, \omega, u_{\alpha, \tau})\|^p_p,
  $$
  which along with \eqref{patup1_21} shows   that,
  for all $\alpha \in [0, \alpha_0]$   and $t \in [\tau, \tau +T]$,
  \be\label{patup1_30}
   \|   v_\alpha (t, \tau ,  \omega, v_{\alpha, \tau } ) \|^2 + 
\int_{\tau  }^t    
\|    v_\alpha (s, \tau ,  \omega, v_{\alpha, \tau } )  \|^p_p ds
 \le c  e^{\lambda (t-\tau)}
\left  (  T+  \|u_{\alpha, \tau} \|^2 + 
   \int_{\tau}^t     \| g(s, \cdot) \|^2    ds
 \right ) .
 \ee
  It follows   from  
   \eqref{patup1_20} and \eqref{patup1_22}-\eqref{patup1_30}
  that,   for all $\alpha \in [0, \alpha_0]$   and $t \in [\tau, \tau +T]$,
$$ 
  \| v_\alpha (t, \tau, \omega, v_{\alpha, \tau}) - u(t,\tau, \omega)  \|^2
  \le e^{c_3 (t -\tau)} \|  v_{\alpha, \tau} -  u_\tau  \|^2
  $$
    \be\label{patup1_31}
  + c_5 \varepsilon e^{c_5 (t -\tau)}
  \left  ( T +
\|u_\tau\|^2 + \|u_{\alpha, \tau} \|^2 + 
   \int_{\tau}^t   \| g(s, \cdot) \|^2   ds
   \right ).
 \ee
 This completes  the proof. 
  \end{proof}

As a consequence of Lemma \ref{atup1} we have the following 
estimates for $u_\alpha (t, \tau, \omega, u_{\alpha, \tau})$.

\begin{cor}
 \label{atup2}
 Suppose 
\eqref{f1}-\eqref{f4} and \eqref{f5} hold. 
Then,  for every $\tau \in \R$,
$\omega \in \Omega$,
$T>0$   and $\varepsilon \in [0, 1)$,
there exists a positive number
 $\alpha_0 = \alpha_0(\tau, \omega, T, \varepsilon)$
 such that for all $\alpha \le \alpha_0$  and $t \in [\tau,
 \tau +T]$,   
 $$
  \| u_\alpha (t, \tau, \omega, u_{\alpha, \tau}) - u(t,\tau, \omega)  \|^2
  \le c  e^{c  (t -\tau)} \|  u_{\alpha, \tau} -  u_\tau  \|^2
  $$
   $$
  + c  \varepsilon e^{c  (t -\tau)}
  \left  ( T +
\|u_\tau\|^2 + \|u_{\alpha, \tau} \|^2 + 
   \int_{\tau}^t   \| g(s, \cdot) \|^2   ds
   \right ), 
$$
where $c$ is a positive   constant independent of  $\tau$, $\omega$, $\varepsilon$
and $\alpha$.  
\end{cor}

\begin{proof}
It follows from  \eqref{patup1_2}  that, 
    for all $\alpha \in [0, \alpha_0]$   and $t \in [\tau, \tau +T]$,
    $$
     \| u_\alpha (t, \tau, \omega, u_{\alpha, \tau}) - u(t,\tau, \omega)  \|^2
     \le 2  \|u_\alpha (t, \tau, \omega, u_{\alpha, \tau})
     -  v_\alpha (t, \tau, \omega, v_{\alpha, \tau})    \|^2
     + 2   \| v_\alpha (t, \tau, \omega, v_{\alpha, \tau}) - u(t,\tau, \omega)  \|^2
    $$
$$  
     \le 2  \left |   e^{\alpha \omega (t)} -1 \right |^2
       \|v_\alpha (t, \tau, \omega, v_{\alpha, \tau})  \|^2
     + 2   \| v_\alpha (t, \tau, \omega, v_{\alpha, \tau}) - u(t,\tau, \omega)  \|^2
    $$
$$  
     \le 2 \varepsilon^2 
       \|v_\alpha (t, \tau, \omega, v_{\alpha, \tau})  \|^2
     + 2   \| v_\alpha (t, \tau, \omega, v_{\alpha, \tau}) - u(t,\tau, \omega)  \|^2,
    $$
    which along with  \eqref{atup1a1},  \eqref{patup1_2}
    and  \eqref{patup1_30}
    completes   the proof. 
    \end{proof}

 Next,  we present uniform estimates of solutions
 with respect to the intensity  $\alpha $  of  noise.
 These estimates are needed for establishing 
 the upper semicontinuity of pullback attractors. 
 By carefully examining the proof of Lemmas \ref{lem1}, \ref{lem3}
 and \ref{lemu1}, we   get     following uniform estimates.
 
\begin{lem}
\label{esal1}
 Suppose  \eqref{f1}-\eqref{f4}  and \eqref{gcon1} hold.
Then for every $\tau \in \R$, $\omega \in \Omega$   
and $D=\{D(\tau, \omega)
: \tau \in \R,  \omega \in \Omega\}  \in \cald$,
 there exists  $T=T(\tau, \omega,  D) \ge 1$ such that 
 for all $t \ge T$ and $\alpha \in [0, 1]$,
$$
\| u_\alpha(\tau, \tau -t,  \theta_{2, -\tau} \omega, u_{\alpha, \tau -t}  ) \|^2 _\hone
\le   c    \int_{-\infty}^0
e^{\lambda s}    e^{2 |\omega (s)|} 
 \left (1+    \| g(s+\tau, \cdot) \|^2  \right )   ds,
$$
 where $u_{\alpha, \tau -t}\in D(\tau -t, \theta_{2, -t} \omega)$ and
  $c$ is a  positive constant  independent of $\tau$, $\omega$, $D$ 
  and $\alpha$.
\end{lem}

 Based on Lemma \ref{esal1}, 
 from  the proof of Lemmas \ref{lem4}  and \ref{lemu2} we 
 get 
 the following  estimates on the tails of solutions.
 
  \begin{lem}
\label{esal2}
Suppose  \eqref{f1}-\eqref{f4}  and \eqref{gcon1} hold.
  Let  $\tau \in \R$, $\omega \in \Omega$   and $D=\{D(\tau, \omega)
: \tau \in \R,  \omega \in \Omega\}  \in \cald $.
Then  for every $\eta>0$,   there exist
    $T=T(\tau, \omega,  D, \eta) \ge 1$
 and $K=K(\tau, \omega, \eta) \ge 1$ such that for all $t \ge T$
 and $\alpha \in [0,1]$,  
 $$
 \int_{|x| \ge K}
 |u_\alpha (\tau, \tau -t, \theta_{2, -\tau} \omega, u_{\alpha, \tau -t} ) (x)|^2  dx \le \eta,
$$
 where $u_{\alpha, \tau -t}\in D(\tau -t, \theta_{2, -t} \omega)$.
\end{lem}

We now   prove  
  the precompactness of the union
of  pullback attractors. 

\begin{lem} 
\label{esal3}
 Suppose  \eqref{f1}-\eqref{f4}  and \eqref{gcon2} hold.
Then for every $\tau \in \R$
and $\omega \in \Omega$,
 the union $\bigcup\limits_{0< \alpha \le 1}
\cala_\alpha (\tau, \omega)$
is precompact in $\ltwo$.
\end{lem}

\begin{proof}
We  only need to show   that,
for every $\varepsilon>0$,  $\tau \in \R$
and $\omega \in \Omega$, 
the set 
$\bigcup\limits_{0< \alpha \le 1}\cala_\alpha (\tau, \omega)$
 has a finite covering of balls of radius less than
$\varepsilon$. 
   Let $B$ be the family of subsets of
   $\ltwo$ given by
   \eqref{upabs4}.
    Then it follows from
   Lemma \ref{esal2}  that  
there exist     $T = T  (\tau, \omega, \varepsilon) \ge 1$ and
$L = L  (\tau, \omega, \varepsilon) \ge 1$ 
such that    for all $t \ge T $
and $\alpha \in (0, 1]$,
\be
 \label{pesal3_1}
\int_{|x| \ge L } 
u_\alpha (\tau, \tau -t, \theta_{2, -\tau} \omega, u_{\alpha, \tau -t} ) (x)|^2  dx 
\le {\frac 12} \varepsilon,
\ee
 where $u_{\alpha, \tau -t}\in B(\tau -t, \theta_{2, -t} \omega)$.
 By \eqref{upabs5a1} and \eqref{pesal3_1},  we get
 from the invariance of $\cala_\alpha$ 
 that, 
   for  each   $\tau \in \R$   and 
   $\omega \in \Omega$,
\be
 \label{pesal3_2}
\int_{|x| \ge L } 
|u (x)|^2  dx 
\le {\frac 12} \varepsilon, \quad \text{   for all }   \   u \in \cala_\alpha(\tau, \omega)
\ \text{ with }  \  0< \alpha \le 1.
\ee
By \eqref{upabs7} we see 
the set  $   \bigcup\limits_{0< \alpha \le 1}
\cala_\alpha (\tau, \omega)$
is bounded in $H^1(Q)$
with $Q=\{x\in \R^n: |x|<L\}$.
Then 
 the compactness of embedding $H^1(Q)
\subseteq  L^2(Q)$ 
implies  that  the set 
$   \bigcup\limits_{0< \alpha \le 1}
\cala_\alpha (\tau, \omega)$
has a finite covering of balls of radii less than
${\frac 1{2}}\varepsilon$ in $L^2(Q)$,
which together with 
\eqref{pesal3_2} 
completes the proof.
\end{proof}

We are now   in a 
position to  present 
the upper semicontinuity 
of pullback attractors for the stochastic 
equation \eqref{rd1}.

\begin{thm} 
\label{upthm1}
 Suppose  \eqref{f1}-\eqref{f4},  \eqref{gcon2}
  and  \eqref{f5} hold.
Then for every $\tau \in \R$
and $\omega \in \Omega$,
 \be\label{upthm1a1}
\lim_{\alpha  \to 0}    { \rm dist}_\ltwo
( \cala_\alpha (\tau, \omega),  \cala_0(\tau)  )
=0.
\ee
 \end{thm}
 
 \begin{proof}
 Let $K_\alpha$ and
 $K_0$ be the families of
 subsets of $\ltwo$ given by
 \eqref{upabs1}
 and \eqref{upabs3}, respectively.
 Then we know $K_\alpha$ is a 
 $\cald$-pullback absorbing  set 
 of $\Phi_\alpha$, and $K_0$
 is a $\cald_0$-pullback
 absorbing set of $\Phi_0$
 in $\ltwo$. 
 By \eqref{upabs1}-\eqref{upabs3a1} we have,
 for every $\tau \in \R$   and $\omega \in \Omega$,
 \be\label{pupthm1_1}
 \limsup_{\alpha \to 0}
 \| K_\alpha (\tau, \omega) \|
 = \limsup_{\alpha \to 0}
  M_\alpha(\tau, \omega)  =
  M_0 (\tau)  = \| K_0 (\tau ) \|.
\ee
Take a sequence 
$\alpha_n \to 0$ and $u_{0,n} \to u_0$ in $\ltwo$.
By  Corollary \ref{atup2} we  get,
 for  every $t \in \R^+$,
  $\tau \in \R$  and   $\omega \in \Omega$, 
\be  \label{pupthm1_2}
\Phi_{\alpha_n} (t, \tau,  \omega, u_{0,n}) \to \Phi (t, \tau,  u_0)
\quad \text{ in } \ \ltwo.
\ee
From \eqref{pupthm1_1}-\eqref{pupthm1_2} we see
that $\Phi_\alpha$ and $\Phi_0$
satisfy conditions \eqref{sem1nd_1}-\eqref{sem1nd}.
On the other hand, 
By Lemma \ref{esal3}
we find that $\cala_\alpha$ also satisfies 
\eqref{sem3}.
Thus \eqref{upthm1a1} follows from 
   Theorem  \ref{semcondn} immediately. 
 \end{proof}

\end{document}